\documentclass[11pt]{article}
\usepackage[ansinew]{inputenc}
\usepackage{graphicx}
\usepackage{color,lscape}
\usepackage{amsmath,amsthm,amsfonts, ams fonts,amssymb,amsmath,amstext,}
\usepackage{latexsym}
\usepackage[authoryear]{natbib}
\usepackage{dsfont} 
\usepackage{paralist} 
\usepackage{multirow}
\usepackage{nicefrac}
\usepackage{url}
\newcommand{\R}{\mathds{R}}

\newcommand{\E}{\mathbb{E}}
\newcommand{\PP}{\mathbb{P}}
\newcommand{\PPn}{\mathbb{P}^0}
\newcommand{\PPe}{\mathbb{P}^1}
\newcommand{\N}{\mathbb{N}}

\newcommand{\f}{\mathcal F}

\renewcommand{\lll}{\mathcal L}

\newcommand{\proban}{(\Omega^0 ,\f^0,\PPn)}
\newcommand{\probae}{(\Omega^1 ,\f^1,\PPe)}

\newcommand{\probao}{(\Omega ,\f,\PP)}
\newcommand{\probpo}{(\Omega' ,\f',\PP')}
\def\probto{\hbox{$(\widetilde{\Omega},\widetilde{\f},\widetilde{\PP})$}}

\newcommand{\toop}{\stackrel{\PP}{\longrightarrow}}
\newcommand{\tols}{~\stackrel{\lll-(s)}{\longrightarrow}~}

\newcommand{\weak}{\stackrel{w}{\longrightarrow}}

\newcommand{\WEE}{\widetilde{\mathbb{E}}}

\renewcommand{\theequation}{\arabic{section}.\arabic{equation}}

\newtheorem{prop}{Proposition}[section]
\newtheorem{ass}[prop]{Assumption}
\newtheorem{exa}[prop]{Example}

\newtheorem{theo}[prop]{Theorem}
\newtheorem{rem}[prop]{Remark}

\begin{document}

\title{Estimating the quadratic covariation of an asynchronously observed semimartingale with jumps}
\author{Markus Bibinger\thanks{Humboldt-Universit\"at Berlin, Institut f\"ur Mathematik, Unter den Linden 6, 10099 Berlin, Germany. E-Mail: bibinger@math.hu-berlin.de. Financial support from the Deutsche Forschungsgemeinschaft via SFB 649 ``Ökonomisches Risiko'', Humboldt-Universität zu Berlin, is gratefully acknowledged.} \quad and \quad Mathias Vetter\thanks{Ruhr-Universit\"at Bochum, Fakult\"at f\"ur Mathematik, 44780 Bochum, Germany. E-mail: mathias.vetter@rub.de. The author is thankful for financial support through the collaborative
research center ``Statistik nichtlinearer dynamischer Prozesse" (SFB 823) of the Deutsche Forschungsgemeinschaft.}}

\maketitle

\begin{abstract}
We consider estimation of the quadratic (co)variation of a semimartingale from discrete observations which are irregularly spaced under high-frequency asymptotics. In the univariate setting, results from \cite{jacod2008} are generalized to the case of irregular observations. In the two-dimensional setup under non-synchronous observations, we derive a stable central limit theorem for the estimator by \cite{hayayosh2005} in the presence of jumps. We reveal how idiosyncratic and simultaneous jumps affect the asymptotic distribution. Observation times generated by Poisson processes are explicitly discussed.\\[.2cm]
{\it Keywords}: 
asynchronous observations, co-jumps, statistics of semimartingales, quadratic covariation\bigskip
%
%{\it AMS 2000 subject classifications}: primary, 60F05, 60G51, 62H20;
%secondary, 62G32, 62M09.
%
\end{abstract}

%\newpage

\section{Introduction} \label{Intro}
\setcounter{equation}{0}
\renewcommand{\theequation}{\thesection.\arabic{equation}}

Estimating the quadratic variation of a semimartingale $X$ probably is one of the main topics in today's high frequency statistics. Starting with the pioneering work of \citet{andeboll1998} and \citet{barnshep2002} on the use of realized volatility (also called realized variance) as a measure for integrated volatility over a trading day, an enormous number of articles has been dedicated to the development of estimation techniques in this area. Historically first is the extension to power variations which allows for a consistent estimation of integrated quarticity as well -- a necessary task when establishing a so-called feasible central limit theorem for realized volatility that allows to construct confidence sets; see \citet{barnshep2004}.

Estimation approaches for deviations from the idealized setting of observing a continuous semimartingale at equidistant times have attracted a lot of attention since then. For models incorporating jumps, for example, integrated volatility does no longer coincide with the quadratic variation of the underlying process, as it comes from the continuous martingale part of $X$ only. Econometricians, however, are typically interested in estimating integrated volatility due to the belief that this quantity reflects cumulative intrinsic risk of an asset whereas jumps come as external shocks. In the presence of jumps the realized volatility as a discretized quadratic variation converges in probability to the entire quadratic variation under high-frequency asymptotics as the maximum distance between successive observation times tends to zero. This motivated estimators which filter out jumps, like bipower variation by \citet{barnshep2004} and truncated realized volatility by \citet{mancini2009}. Another topic is the treatment of additional microstructure noise in the data; among various proposals see e.g.\ \citet{zhangetal2005}, \citet{barndetal2008} or \citet{podovett2009}. Non-regular observation times have been discussed in various situations: In the univariate context, limit theorems under irregular sampling schemes have been derived both in case of deterministic and random observations times; see e.g.\ \citet{myklzhan2009}, \citet{hayjacyos2011} or \citet{fukarose2012}. In multi-dimensional settings, asynchronicity comes into play which makes the situation more complicated. Let us mention here the approach involving overlapping intervals by \citet{hayayosh2005} and the concept of refresh times from \citet{barndetal2011}.

Where typically less focus has been laid on is statistical inference on the entire quadratic variation of $X$ when jumps are present, though the latter is not only of some importance in economics as a measure of risk comprising jumps and volatility, but also a central quantity in stochastic analysis. Asymptotics in the case of equidistant observations of $X$ are provided as a special case in \citet{jacod2008} who focuses on a number of functionals of semimartingale increments. A similar result for L\'{e}vy processes dates back to \citet{jacodprotter}. Results on an estimator for the quadratic variation when jumps and noise are present are given in \citet{jacodetal2010} for their pre-averaging estimator. Apart from that, at least to the best of our knowledge, no work has dealt with central limit theorems on the entire quadratic (co)variation of $X$, and in particular very few is known in the framework of non-regularly spaced data.

We aim at filling this gap to a certain extent. In a first step, we generalize the asymptotic theory from \citet{jacod2008} on realized volatility and equidistant observations to non-equidistant (univariate) schemes. As a basis for the more involved situations we illuminate how proofs of limit theorems work for general semimartingales in the vein of \citet{podovett2010} who explained limit theorem for the continuous case. In a two-dimensional setting the quadratic covariation
\begin{align}\label{qcv}[X^{(1)},X^{(2)}]_t=\int_0^t\rho_s\sigma_s^{(1)}\sigma_s^{(2)}\,ds+\sum_{s\le t}\Delta X_s^{(1)}\Delta X_s^{(2)}\,\end{align}
is the sum of the integrated covolatility and the sum of products of simultaneous jumps (called co-jumps). The asymptotic theory for co-jumps entails new intriguing attributes and provides deeper insight in the multi-dimensional asymptotic properties of standard estimators.

For non-synchronous observations of continuous It\^o semimartingales, the prominent estimator by \citet{hayayosh2005} for integrated covolatility attains the minimum variance in the general semiparametric situation. We discuss its properties extended to the case of observing a general It\^o semimartingale possibly admitting jumps. Consistency for the entire quadratic covariation is established under mild regularity assumptions. We deduce sufficient conditions on the observation times design to establish a central limit theorem. In particular, we illustrate the formal expressions for the important (and included) setup of exogenous observation times generated by homogenous Poisson processes.

The paper is organized as follows: We review the one-dimensional results by \citet{jacod2008} for realized volatility in Section \ref{sec:2}. The first generalization to non-equidistant observation times is pursued in Section \ref{sec:3}. In Section \ref{sec:4} we develop the asymptotic theory for the Hayashi-Yoshida estimator and non-synchronous two-dimensional observations. The case of Poisson sampling is treated as an explicit example. Section \ref{sec:5} demonstrates the finite sample accuracy in Monte Carlo simulations. The proofs are given in the Appendix. 

\section{The baseline case: univariate regular observations} \label{sec:2}
\setcounter{equation}{0}
\renewcommand{\theequation}{\thesection.\arabic{equation}}

Let us start with revisiting the central limit theorem for realized variance in the presence of jumps for the regular univariate setting which has been found by \cite{jacod2008}. Suppose in the sequel that $X$ is a one-dimensional It\^o semimartingale on $\probao$ of the form
\begin{align} \label{univX}
X_t = X_0 +  \int_0^t b_s\, ds + \int_0^t \sigma_s\, dW_s + \int_0^t \int_{\R} \kappa (\delta(s,z)) (\mu-\nu)(ds,dz) + \int_0^t \int_{\R} \kappa'(\delta(s,z)) \mu(ds,dz),
\end{align}
where $W$ is a standard Brownian motion, $\mu$ is a Poisson random measure on $\R^+ \times \R$, and the predictable compensator $\nu$ satisfies $\nu(ds,dz) = ds \otimes \lambda(dz)$ for some $\sigma$-finite measure $\lambda$ on $\R$ endowed with the Borelian $\sigma$-algebra. $\kappa$ denotes a truncation function with $\kappa(x)=x$ on a neighbourhood of zero and we set $\kappa'(x) = x - \kappa(x)$, to separate the martingale part of small jumps and the large jumps. $\kappa$ is assumed to be continuous here, which helps to simplify notation and further regularity conditions, and with compact support. We impose the following fairly general structural assumptions on the characteristics of $X$. 
\begin{ass} \label{coeff}
The processes $b_s,\sigma_s$ and $s \mapsto \delta(s,z)$ are continuous. Furthermore, we have $|\delta(s,z)| \leq \gamma(z)$ for some bounded positive real-valued function $\gamma$ which satisfies $\int (1 \wedge \gamma^2(z)) \lambda(dz) < \infty$. 
%The volatility process $\sigma_s$ belongs to a Hölder ball of order $\alpha\in(0,1]$ and radius $R>0$, i.e. $\sigma\in C^\alpha(R)$ with
%\begin{align*}&C^{\alpha}(R)=\{f\in C^{\alpha}([0,1],\R)|\|f\|_{C^{\alpha}}\le R\}\text{ where }\|f\|_{C^{\alpha}}\:=\|f\|_{\infty}+\sup_{x\ne y}{\frac{|f(x)-f(y)|}{|x-y|^{\alpha}}}\,.\end{align*}
\end{ass}
Our target of inference is the quadratic variation of the semimartingale $X$ at time $0 < t \leq 1$ which becomes
\[
[X,X]_t = \int_0^t \sigma_s^2 ds + \sum_{s \leq t} (\Delta X_s)^2,
\] 
the sum of the integrated variance and the sum of squared jumps, in the setting above. Here, $\Delta X_s = X_s - X_{s-},X_{s-}=\lim_{t\rightarrow s,t<s}X_t$, denotes the possible jump at time $s$. In the baseline case of equidistant observations, that is we observe $X$ at the regular times $i /n$, $i=0, \ldots , \lfloor nt \rfloor$, \cite{jacod2008} establishes a stable central limit theorem for the natural estimator realized variance. With $\Delta_i^n X = X_{i /n} - X_{(i-1)/n}$, the latter term is defined as
\[
RV_t^n = \sum_{i=1}^{\lfloor nt \rfloor} (\Delta_i^n X)^2.
\]
Before we state the result, let us shortly recall the notion of stable convergence. A family of random variables $(Y_n)$ defined on $\probao$ is said to converge $\mathcal F$-stably in law to $Y$ defined on an extended space $\probto$, if 
\[
\E[g(Y_n)S] \to \widetilde \E[g(Y)S]
\]
for all bounded, continuous $g$ and all bounded $\mathcal F$-measurable random variables $S$. For background information on the notion of stable (weak) convergence we refer interested readers to \cite{jacoshir2003}, \cite{jacodprotter}, \cite{jacod1} and \cite{podovett2010}.
%Before we state his result, we need to introduce some additional notation since limiting variables in this context are naturally defined on an extension of the original probability space. 

In our context, the limiting variable depends on auxiliary random variables. We therefore consider a second space $\probpo$ supporting a standard Brownian motion $W'$ and a sequence $(U'_p)_{p \geq 1}$ of standard normal variables, all mutually independent. The extended space $\probto$ is then given by the (orthogonal) product of the two spaces where all variables above are extended to it in the canonical way. The limiting variables in the central limit theorem for quadratic variation are then defined as follows: Let $(S_p)_{p\ge 1}$ be a sequence of stopping times exhausting the jumps of $X$ and set
\[
Z_t = 2 \sum_{p: S_p \leq t} \Delta X_{S_p} \sigma_{S_p} U'_p \quad \mbox{and} \quad V_t = \sqrt 2 \int_0^t \sigma_s^2 dW'_s.
\]
The stable limit theorem for quadratic variation adopted from \cite{jacod2008} now reads as follows:
\begin{theo} \label{theo1}
Suppose that $X$ is a one-dimensional It\^o semimartingale with representation (\ref{univX}) for which Assumption \ref{coeff} is satisfied. Then for each $0 < t \leq 1$ we have the $\mathcal F$-stable central limit theorem
\begin{align}\label{cltuni}
n^{1/2} (RV_t^n - [X,X]_t) \tols V_t + Z_t.
\end{align}
\end{theo}
\begin{rem} \rm
Even though $Z_t$ might depend on the particular choice of the stopping times, it is shown in \cite{jacod2008} that its $\f$-conditional law does not. By definition of stable convergence, this is all that matters. Note also that the result above only holds for a fixed $t > 0$, but not in a functional sense, except $X$ is continuous. This is due to the fact that a large jump at time $t$ is by definition included in $[X,X]_t$, but usually not in $RV_t^n$, as the latter statistic only counts increments up to time $\lfloor nt \rfloor /n$. For a fixed $t$, this is not relevant, as the expectation of large jumps close to time $t$ is small, but in a process sense this issue becomes important. One can account for this fact by subtracting $[X,X]_{\lfloor nt \rfloor /n}$ in (\ref{cltuni}) instead, however.\qed
\end{rem}
We give a proof of Theorem \ref{theo1} in Appendix \ref{app2} and \ref{app2b}, basically for two reasons: First, the analogue of Theorem \ref{theo1} is only a special case of the much more general discussion in \cite{jacod2008}, and we believe that it is interesting to highlight how proofs of stable central limit theorems concerned with jumps work in this special (but nevertheless important) situation. In this sense, the first part of this paper can be understood as a follow-up to \cite{podovett2010} where the focus was on explaining limit theorems for continuous semimartingales. Second, the proof serves as foundation for all other setups where we employ the results provided for the baseline case discussed in this section.

Throughout the paper, we restrict ourselves to continuous $\sigma$. This condition can be weakened in the sense that it might be some It\^o semimartingale itself. We refer to \cite{jacod2008} to an extension of \eqref{cltuni} allowing even for common jumps of $\sigma$ and $X$, in which the limit $Z_t$ is slightly more complicated. Since we shall focus on the effects of irregular sampling, and also on the impact of jumps on the Hayashi-Yoshida estimator in the multivariate case, which furnish several new interesting effects, we believe this slight simplification helps to keep the asymptotic results readable and clear.  

\section{Asymptotics for irregular sampling schemes} \label{sec:3}
\setcounter{equation}{0}
\renewcommand{\theequation}{\thesection.\arabic{equation}}
The situation changes when the observations do not come at regular times anymore. In general, at stage $n$ one observes a one-dimensional process $X$ at arbitrary times $0= t_0^n < t_1^n < \ldots$, which may either be deterministic or random (stopping) times, and a further distinction in the random case regards independent and endogenous sampling schemes. The latter are by far the most complicated, and it is well-known that already in the continuous case central limit theorems become non-standard for observations based e.g.\ on hitting times of $X$; see \cite{fukarose2012} and related papers. For this reason, we restrict ourselves in this work to either deterministic observations times or those coming from independent random variables. Even in this case, it is hard to derive asymptotics in general, and this becomes particularly virulent in the multi-dimensional framework discussed in the next section. 

We use the notation \(m^n_-(t) = \max \{i: t_i^n \leq t\}, \tau^n_-(t)=\max \{t_i^n: t_i^n \leq t\}\) and \(m^n_+(t) = \min \{i: t_i^n \ge t\}, \tau^n_+(t)=\min \{t_i^n: t_i^n \ge  t\}\) for an arbitrary $0 \leq t \leq 1$, referring to the number of observations around time $t$ and to the previous and next ticks.
A necessary condition in order to infer on the quadratic variation of $X$ is to secure that the mesh of the observation times $\pi_n = \max \{|t_i^n - t_{i-1}^n| | i =1, \ldots, m_-^n(1)\}$ tends to zero (in probability) as $n$ increases. Standard results from stochastic analysis then ensure consistency of realized variance as an estimator for the quadratic variation, which becomes 
\[
RV_t^n = \sum_{i=1}^{m^n_{-}(t)} |\Delta_i^n X|^2 \toop [X,X]_t
\]
in this context. Here we have set $\Delta_i^n X = X_{t_i^n} - X_{t_{i-1}^n}$ again.

In order to derive a central limit theorem for $RV_t^n$, we need sharper bounds on the order of $\pi_n$ as well as two regularity conditions on increments of the observations times. The first assumption is concerned with the variance due to the continuous martingale part, whereas the second one is about local regularity around possible jump times. It looks rather complicated, but reflects precisely what is needed to prove stable convergence later on.
\begin{ass} \label{scheme}
Suppose that the variables $t_i^n$ are stopping times which satisfy $\E[\pi_n^q] = o(n^{-\alpha})$ for all $q \geq 1$ and any $0 < \alpha < q$. Furthermore, we assume 
\begin{itemize}
	\item[(i)] that there exists a continuously differentiable function $G:[0,1] \to [0,\infty)$, such that the convergence
\begin{align} \label{Hexist}
G(t) = \lim_{n \to \infty} G_n(t) := \lim_{n \to \infty} n \sum_{i=1}^{m^n_-(t)} (t_i^n - t_{i-1}^n)^2
\end{align}
holds pointwise (in probability)
\item[(ii)] that for any $0 < t \leq 1$ and any $k \in \N$ we have convergence of
\begin{align} \label{relterm}
\int_{[0,t]^k} g(x_1, \ldots, x_k) \E\Big[\prod_{p=1}^k h_p(n(\tau^n_+(x_p) - \tau^n_-(x_p)))\Big] dx_k \ldots dx_1
\end{align}
to
	\begin{align} \label{locexist} 
	\int_{[0,t]^k} g(x_1, \ldots, x_k) \prod_{p=1}^k \int_\R h_p(y) \Phi(x_p,dy) dx_k \ldots dx_1
	\end{align}
as $n\rightarrow\infty$, where the $\Phi(x,dy)$ denote a family of probability measures on $[0,\infty)$ with uniformly bounded first moment and $g$ and $h_p$, $p=1, \ldots, k$, are bounded continuous functions.
%	\item[(ii)] that one of the two following local conditions is satisfied:
%	\begin{itemize}
%	\item[(ii-a)]For any $s$, $0<s\le t$,
%	\begin{subequations}
%	\begin{align} \label{locexist1} 
%	n(\tau_+^n(s) - \tau_-^n(s)) \to g(s)
%	\end{align}
%	as $n\rightarrow\infty$ for a deterministic function $g$.
%	\item[(ii-b)]For a random variable $U$ uniformly distributed on $[0,t]$,
%	\begin{align} \label{locexist2} 
%	n(\tau_+^n(U) - \tau_-^n(U)) \stackrel{\mathcal{L}}{\longrightarrow} \eta
%	\end{align}
%	as $n\rightarrow\infty$ for a random variable $\eta$ independent of $U$.
%	\end{subequations}
%	\end{itemize} 
\end{itemize}
\end{ass}

Note in Assumption \ref{scheme} (ii) that the expectation of products in \eqref{relterm} becomes a product of expectations in \eqref{locexist}. This means that after standardization the lengths of the intervals around the (jump) times $x_p$ converge to independent variables, whose distributions may in general depend on $x_p$. The latter property reflects for example that there might be periods in which observations come more often than in others. 

\begin{exa} \label{detsche} \rm Suppose that the sampling scheme is deterministic with $t_i^n = f(i/n)$ for some strictly isotonic, deterministic function $f: [0,1] \to [0,1]$. If $f$ is continuously differentiable, then Assumption \ref{scheme} is satisfied with the deterministic limits
\[
G(t) = \lim_{n \to \infty} n \sum_{i=1}^{m^n_-(t)} (t_i^n - t_{i-1}^n)^2 = \lim_{n \to \infty} \frac 1n \sum_{i=1}^{m^n_-(t)} (f'(\xi_i^n))^2 = \int_0^t (f'(x))^2 dx.
\]
In order to prove the representation \eqref{locexist} set $\eta(x)=f^{\prime}(f^{-1}(x))$. Since the design is deterministic, the expectation in (\ref{relterm}) can be dropped and we obtain 
\[
n(\tau_+^n(x) - \tau_-^n(x)) = n\left[f\left(\left\lceil n f^{-1}(x)\right\rceil/n) - f(\left\lfloor  n f^{-1}(x) \right\rfloor/n \right) \right] \to f'(f^{-1}(x)) = \eta(x).
\]
Therefore \eqref{locexist} holds with the deterministic $\Phi(x,dy) = \delta_{\eta(x)}(dy).$ The bound on $\E[\pi_n^q]$ is trivially satisfied as well. \qed
\end{exa}
\begin{exa} \label{poische} \rm Alternatively, one might want to work with a random observation scheme. Classical is Poisson sampling, where $(N^n)_t$ is a Poisson process with intensity $n\lambda$ for $\lambda > 0$ fixed and each $n$, and the observations time $t_i^n$ is equivalent to the time of the $i$-th jump of $N^n$. In this case, we have
\[
n \sum_{i=1}^{m^n_-(t)} (t_i^n - t_{i-1}^n)^2 = n \sum_{i=1}^{\left\lceil n\lambda t \right\rceil} (t_i^n - t_{i-1}^n)^2 + o_p(1) = 2t/\lambda + o_p(1),
\]
since the $t_i^n - t_{i-1}^n$ form a sequence of i.i.d.\ $\exp(n\lambda)$-variables. We have used both Lemma 8 in \cite{hayayosh2008}, which states that $\E[\pi_n^q] = o(n^{-\alpha})$ is satisfied, and arguments from the proof of Lemma 10, which show that $m^n_-(t)$ is close to $\left\lceil n\lambda t \right\rceil$, to obtain the first relation above. Let us now derive the limit of 
\[
\E\Big[\prod_{p=1}^k h_p(n(\tau^n_+(x_p) - \tau^n_-(x_p)))\Big]
\]
for any fixed $x_1, \ldots, x_k$, and we start with $k=1$. First, due to memorylessness $n(\tau_+^n(x_1) - x_1) \sim \exp(\lambda)$. On the other hand, a standard result in renewal theory (see e.g.\ \cite{cox1970}; page 31) gives the distribution of the backward recurrence time of the Poisson process: 
\[
\PP(n (x_1- \tau_-^n(x_1)) \leq u) = \begin{cases} 1, \quad &u=n x_1, \\ 1- e^{-\lambda u}, \quad &0 < u < nx_1.\end{cases}
\]
Therefore, $n (x_1- \tau_-^n(x_1)) \weak \exp(\lambda)$ as $n\rightarrow\infty$, and from the strong Markov property, which secures independence of the two summands, we have $n(\tau_+^n(x_1) - \tau_-^n(x_1)) \weak \Gamma(2,\lambda)$. Similarly, for a general $k$, one can show that the $n(\tau_+^n(x_k) - \tau_-^n(x_k))$ are asymptotically independent, and all sequences of random variables obviously have the same limiting distribution. Condition \eqref{locexist} is therefore valid with $\Phi(x,dy)$ being the distribution of a $\Gamma(2,\lambda)$ variable for all $x$.\qed
\end{exa}
\begin{exa} \label{detirsche} \rm As a third example we consider a deterministic irregular scheme with a truly random limiting distribution $\Phi(x,dy)$. Consider the sequence of observation times $t_i^n=i/n,i=2,4,\ldots,$ for even numbers and $t_i^n=(i+\alpha)/n,i=1,3,\ldots,$ for odd numbers with some $0<\alpha<1$. For fixed $x_1$, interval lengths $n(\tau_+(x_1)-\tau_-(x_1))$ can alternate between $(1+\alpha)$ and $(1-\alpha)$ and do not converge. This is where a random limit comes into play: Let us again discuss $k=1$ in detail. Setting $[0,t]=A_n \cup B_n$, where $A_n$ denotes the subset on which $n(\tau_+(x_1)-\tau_-(x_1)) = 1+\alpha$ and  $B_n$ the one with $n(\tau_+(x_1)-\tau_-(x_1)) = 1-\alpha$, we obtain from continuity of $g$
\begin{align*}
\int_{[0,t]^k} g(x_1) h_1(n(\tau^n_+(x_1) - \tau^n_-({{x_1}}))) dx_1 &= h_1(1+\alpha) \int_{A_n} g(x_1) dx_1 +  h_1(1-\alpha) \int_{B_n} g(x_1) dx_1 \\ &\sim \frac 1t (h_1(1+\alpha) \lambda(A_n) + h_1(1-\alpha) \lambda(B_n)) \int_0^t g(x_1) dx_1 \\ &\to (h_1(1+\alpha) (1+\alpha)/2 + h_1(1-\alpha) (1-\alpha)/2) \int_0^t g(x_1) dx_1,
\end{align*}
where $\lambda$ denotes the Lebesgue measure. Thus, $\Phi(x,dy)$ is again independent of $x$ and has two atoms taking the value $(1+\alpha)$ with probability $(1+\alpha)/2$ and the value $(1-\alpha)$ with probability $(1-\alpha)/2$. A generalization to arbitrary $k$ is straightforward. Note also that the condition on $\pi_n$ is satisfied by definition and that \eqref{Hexist} holds with $G(t) = (1+\alpha^2)t$. \qed
\end{exa}
Let us now extend \eqref{cltuni} to this framework. Thereto, denote with $\proban$ a probability space on which $X$ is defined, and we assume all observation times $t_i^n$ to live on $\probae$. We can now define $\probao$ as the product space of these two, while $\probpo$ is defined similarly as before, but it is assumed to accommodate independent random variables $(\eta(x))_{0 \leq x \leq t}$ as well, with distribution $\Phi(x,dy)$ as in (\ref{locexist}).  $\probto$ finally is the orthogonal product of the latter two spaces again. 
\begin{theo} \label{theo1b}
Suppose that $X$ is a one-dimensional It\^o semimartingale with representation (\ref{univX}) for which Assumption \ref{coeff} is satisfied. If also Assumption \ref{scheme} on the observation scheme holds, then for each $0 < t \leq 1$ we have the $\mathcal F^0$-stable central limit theorem
\begin{align}\label{cltunib}
n^{1/2} (RV_t^n - [X,X]_t) \tols \widetilde V_t + \widetilde Z_t, 
\end{align}
where 
\[\widetilde V_t = \sqrt 2 \int_0^t \sigma_s^2 (G'(s))^{1/2} dW'_s\]
and
\[\widetilde Z_t = 2 \sum_{p: S_p \leq t} \Delta X_{S_p} \big(\eta(S_p)\big)^{1/2} \sigma_{S_p} U'_p.\]
Here, the $S_p$ are stopping times exhausting the jumps of $X$ and the $(U'_p)$ are i.i.d.\ standard normal on $\probpo$ as before.
\end{theo}
This theorem has already been known in the literature, if $X$ is a continuous process; see e.g.\ the survey by \cite{myklzhan2012}.
\begin{rem} \rm
Both limiting processes look similar to the ones obtained in Theorem \ref{theo1}, apart from different standardizations due to irregular sampling. What is interesting, however, is the nature of the scaling in the part due to jumps. The schemes considered in Example \ref{detsche} are locally regular, which leads to deterministic $\eta(S_p)$ as well. On the other hand, both the Poisson sampling and the deterministic design in Example \ref{detirsche} show local irregularities, resulting in random (but time-homogeneous) limits $\eta(S_p)$. Nevertheless, we still have regularity on a global level even for these sampling schemes, leading to a deterministic limit of $G_n(t)$ in all three cases. \qed
\end{rem}

%Erwähnen, dass Jacod auch multivariate Resultate behandelt. 
\section{Asymptotics in the multivariate case} \label{sec:4}
\setcounter{equation}{0}
\renewcommand{\theequation}{\thesection.\arabic{equation}}
This section is devoted to non-synchronous discrete observations of a multi-dimensional It\^{o} semimartingale with jumps. It is informative to stick to a two-dimensional setting and an underlying semimartingale of similar form as \eqref{univX} before:
\begin{align}\notag
X_t &= \left(X^{(1)},X^{(2)}\right)_t^{\top}\\
&=  \label{multivX}X_0 +  \int_0^t b_s\, ds + \int_0^t \sigma_s\, dW_s + \int_0^t \int_{\R^2} \kappa (\delta(s,z)) (\mu-\nu)(ds,dz) + \int_0^t \int_{\R^2} \kappa'(\delta(s,z)) \mu(ds,dz),
\end{align}
where $W=\left(W^{(1)},W^{(2)}\right)^{\top}$  denotes a two-dimensional standard Brownian motion, and we assume without loss of generality
\[
\sigma_s = \begin{pmatrix} \sigma_s^{(1)} & 0 \\ \rho_s \sigma_s^{(2)} & \sqrt{1-\rho_s^2} \sigma_s^{(2)}\end{pmatrix} ~\mbox{such that} ~~\sigma_s\sigma_s^{\top}=\begin{pmatrix} \big(\sigma_s^{(1)}\big)^2 & \rho_s \sigma_s^{(1)}\sigma_s^{(2)}  \\ \rho_s \sigma_s^{(1)}\sigma_s^{(2)} & \big(\sigma_s^{(2)}\big)^2\end{pmatrix},
\]
while the other characteristics are defined analogously to Section \ref{sec:2} with two-dimensional jump measures. Denote with $\|\,\cdot\,\|$ the spectral norm. We develop a theory for general jump measures comprising co-jumps ($X^{(1)}$ and $X^{(2)}$ jump at the same time) and idiosyncratic jumps of the components.

We investigate the estimator by \cite{hayayosh2005}, called HY-estimator in the following, under the influence of jumps. The HY-estimator has been proposed and is well-studied for integrated covolatility estimation from asynchronous observations of a continuous It\^{o} semimartingale; see \cite{hayayosh2008} and \cite{hayayosh2011}. Our structural hypothesis for the characteristics of $X$ reads similar as Assumption \ref{coeff} in Section \ref{sec:2}:
\begin{ass} \label{coeff2}
 Assume that $b_s,$ $\sigma_s^{(1)}$, $\sigma_s^{(2)}$, $\rho_s$ and $s\mapsto \delta(s,z)$ are continuous and that $\|\delta(s,z)\| \leq \gamma(z)$ for a bounded positive real-valued function $\gamma$ which satisfies $\int (1 \wedge \gamma^2(z)) \lambda(dz) < \infty$. 
\end{ass}
By It\^{o} isometry, we may expect that in the presence of jumps the HY-estimator is suitable for estimating the entire quadratic covariation \eqref{qcv}. Yet, there are several open questions which we address in this section and an asymptotic distribution theory of the HY-estimator with jumps is unexplored territory.
\subsection{Discussion of the HY-estimator and notation}
The HY-estimator is the sum of products of increments with overlapping observation time instants:
\begin{subequations}
\begin{align}\label{HY}\widehat{\Big[ X^{(1)}, X^{(2)}\Big]}_t^{(HY),n}=\sum_{t_i^{(1)}\le t}\,\sum_{t_j^{(2)}\le t}\Big(X^{(1)}_{t_i^{(1)}}-X^{(1)}_{t_{i-1}^{(1)}}\Big)\Big(X^{(2)}_{t_j^{(2)}}-X^{(2)}_{t_{j-1}^{(2)}}\Big)1_{\Big\{\min{\Big(t_{i}^{(1)},t_{j}^{(2)}\Big)}>\max{\Big(t_{i-1}^{(1)},t_{j-1}^{(2)}\Big)}\Big\}},\end{align}
when $X^{(l)}, l=1,2,$ is observed at times $t_i^{(l)}$. In the sequel, we introduce the notion of several interpolation functions and sequences dependent on the observation times. Let $\pi_n=\max_{i,l}\{|t_i^{(l)}-t_{i-1}^{(l)}|\}$ denote the mesh.
We define
\begin{align*}\tau_+^{(l)}(s)&=\min_{i\in\{0,\ldots,n_l\}}{\Big(t_i^{(l)}|t_i^{(l)}\ge s\Big)}, \quad m_+^{(l)}(s)=\min{\Big(i|t_i^{(l)}\ge s\Big)}\quad\mbox{and}\\
\tau^{(l)}_{-}(s)&=\max_{i\in\{0,\ldots,n_l\}}{\Big(t_{i}^{(l)}|t_i^{(l)}\le s\Big)}, \quad m_-^{(l)}(s)=\max{\Big(i|t_i^{(l)}\le s\Big)}\end{align*}
for $l=1,2$, and $s\in[0,1]$. Let us further introduce the shortcuts
\[\tau_{++}^{(l,r)}(s)=\tau_+^{(l)}\Big(\tau_+^{(r)}(s)\Big)\quad \mbox{and} \quad \tau_{--}^{(l,r)}(s)=\tau_-^{(l)}\Big(\tau_-^{(r)}(s)\Big)\,\]
and $\tau_{++}^{(r,l)}(s), \tau_{--}^{(r,l)}(s)$, analogously.
A synchronous grid serving as a reference scheme is given by the sequence of refresh times
\begin{align*}T_k=\max{\Big(\tau_+^{(1)}(T_{k-1}),\tau_+^{(2)}(T_{k-1})\Big)}, ~k=0,\ldots,M^n(1),\end{align*}
with the convention $T_{-1}=0$ and where we denote with $M^n(t)$ the number of refresh times smaller than or equal to $t\in[0,1]$. Each increment $T_k-T_{k-1}$ thus is the waiting time until both components of $X$ have been observed again. The use of refresh times is adopted from \cite{barndetal2011} where the same synchronous scheme is employed in a more general way. For notational convenience, indices referring to dependence on $n$ for sampling times are often suppressed in the multi-dimensional setup.

Based on telescoping sums, the HY-estimator \eqref{HY} can be rewritten
\begin{align}\label{varHY1}\widehat{\Big[ X^{(1)}, X^{(2)}\Big]}_t^{(HY),n}&=\sum_{i=1}^{m_-^{(1)}(t)-1} \Big(X^{(1)}_{t_i^{(1)}}-X^{(1)}_{t_{i-1}^{(1)}}\Big)\Big(X^{(2)}_{\tau_+^{(2)}(t_i^{(1)})}-X^{(2)}_{\tau_-^{(2)}(t_{i-1}^{(1)})}\Big)+\mathcal{O}_p(\pi_n)\\
&\label{varHY2}=\sum_{j=1}^{m_-^{(2)}(t)-1}  \Big(X^{(2)}_{t_j^{(2)}}-X^{(2)}_{t_{j-1}^{(2)}}\Big)\Big(X^{(1)}_{\tau_+^{(1)}(t_j^{(2)})}-X^{(1)}_{\tau_-^{(1)}(t_{j-1}^{(2)})}\Big)+\mathcal{O}_p(\pi_n)\\
&\label{varHY3}=\sum_{k=1}^{M^n(t)-1}\Big(X_{\tau_+^{(1)}(T_k)}^{(1)}-X_{\tau_-^{(1)}(T_{k-1})}^{(1)}\Big)\Big(X_{\tau_+^{(2)}(T_k)}^{(2)}-X_{\tau_-^{(2)}(T_{k-1})}^{(2)}\Big)+\mathcal{O}_p(\pi_n).\end{align}
\end{subequations}
The $\mathcal{O}_p(\pi_n)$ terms in \eqref{varHY1}--\eqref{varHY3} are only due to possible end effects at time $t$. Apart from this, the above equalities hold exactly. Illustrations \eqref{varHY1}--\eqref{varHY3} reveal that the estimation error of the HY-estimator can be decomposed in the one of a usual synchronous-type realized covolatility and an additional error induced by non-synchronicity and interpolations.
To simplify notation a bit, we write from now on
\begin{align*}\Delta_i^{n_1}X^{(1)}=\Big(X^{(1)}_{t_i^{(1)}}-X^{(1)}_{t_{i-1}^{(1)}}\Big), \quad \Delta_j^{n_2}X^{(2)}=\Big(X^{(2)}_{t_j^{(2)}}-X^{(2)}_{t_{j-1}^{(2)}}\Big), \quad \Delta_k^nX^{(l)}=\Big(X^{(l)}_{T_k}-X^{(l)}_{T_{k-1}}\Big), ~l=1,2,\end{align*}
and for previous and next-tick interpolations with respect to the refresh time scheme
\begin{align*}\Delta_k^{+,n}X^{(l)}=\Big(X^{(l)}_{\tau_+^{(l)}(T_k)}-X^{(l)}_{T_k}\Big), \quad \Delta_k^{-,n}X^{(l)}=\Big(X^{(l)}_{T_{k-1}}-X^{(l)}_{\tau_-^{(l)}(T_{k-1})}\Big), ~l=1,2.\end{align*}
Also, we denote with $\Delta_k^n$ the refresh time instants $(T_k-T_{k-1})$ and $\Delta_k^{+,n_l}=(\tau_+^{(l)}(T_k)-T_k)$ are the next- and $\Delta_k^{-,n_l}=(T_{k-1}-\tau_-^{(l)}(T_{k-1}))$ the previous-tick interpolations.

When decomposing $X$ in different terms by the continuous part, jumps and cross terms, we can use any of the illustrations \eqref{HY}--\eqref{varHY3} to analyze those terms. Therefore, to gain deeper insight and to get used to the notation, let us delve into the different ways to illustrate and construct the HY-estimator:
\begin{itemize}
\item[\eqref{HY}]This is the original idea to sum all products of increments, belonging to time intervals between adjacent observations which have a non-empty intersection.
\item[\eqref{varHY1}]Trace out all increments of $X^{(1)}$ and sum up products with the interpolated increments of $X^{(2)}$: \(\Delta_i^{n_1}X^{(1)}\Big(X^{(2)}_{\tau_+^{(2)}(t_i^{(1)})}-X^{(2)}_{\tau_-^{(2)}(t_{i-1}^{(1)})}\Big),~i=1,\ldots,m_-^{(1)}(t)-1.\)
\item[\eqref{varHY2}]Trace out all increments of $X^{(2)}$ and sum up products with the interpolated increments of $X^{(1)}$: \(\Delta_j^{n_2}X^{(2)}\Big(X^{(1)}_{\tau_+^{(1)}(t_j^{(2)})}-X^{(1)}_{\tau_-^{(1)}(t_{j-1}^{(2)})}\Big),~j=1,\ldots,m_-^{(2)}(t)-1.\)
\item[\eqref{varHY3}]Consider the refresh times grid and sum up products of interpolated increments of $X^{(1)}$ and $X^{(2)}$:\\ \(\Big(\Delta_k^{+,n}X^{(1)}+\Delta_k^n X^{(1)}+\Delta_k^{-,n}X^{(1)}\Big)\Big(\Delta_k^{+,n}X^{(2)}+\Delta_k^n X^{(2)}+\Delta_k^{-,n}X^{(2)}\Big), ~k=1,\ldots,M^n(t)-1.\) At least one of the previous-tick and one of the next-tick interpolations equal zero.
\end{itemize}
\begin{figure}[t]
\centering
\fbox{
\includegraphics[width=16cm]{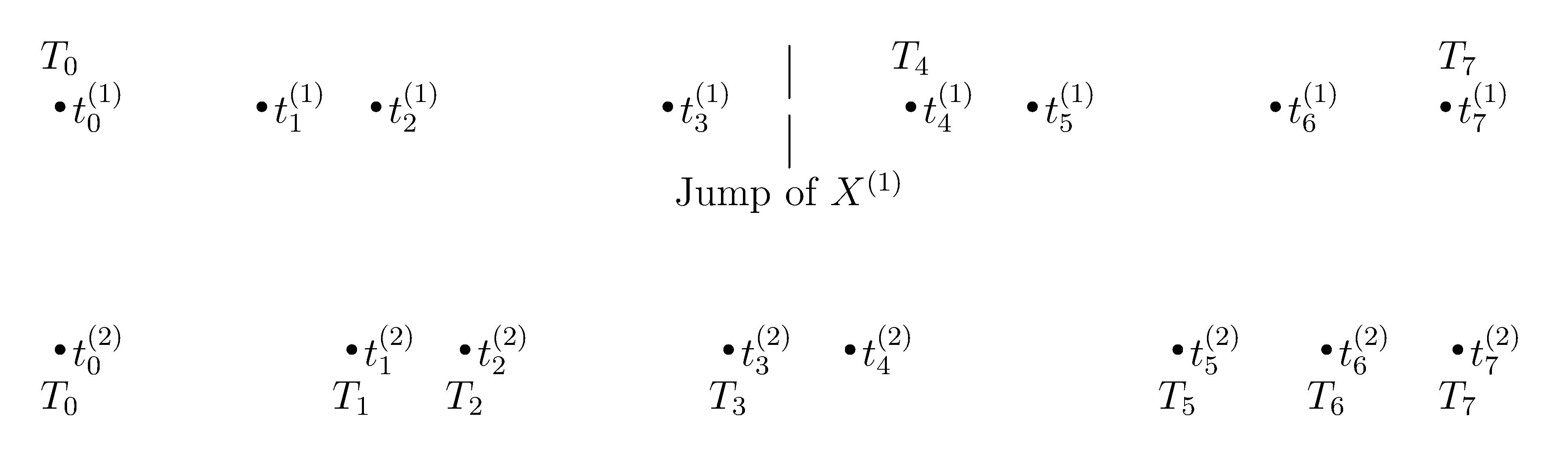}}
\caption{\label{fig1}Example of observation times allocation.}
\end{figure}
\begin{exa} \rm
To illuminate the transformations between \eqref{HY}--\eqref{varHY3} by rearranging addends, particularly in the presence of a jump, we examine a small example displayed in Figure \ref{fig1}. Focusing on the increment $\Delta_4^{n_1}X^{(1)}$ that contains a jump, \eqref{HY} tells us that this increments is considered in the addends
\begin{align*}\Delta_4^{n_1}X^{(1)}\Delta_3^{n_2}X^{(2)}+\Delta_4^{n_1}X^{(1)}\Delta_4^{n_2}X^{(2)}+\Delta_4^{n_1}X^{(1)}\Delta_5^{n_2}X^{(2)}.\end{align*}
If we start with illustration \eqref{varHY1}, we directly obtain
\begin{align*}\Delta_4^{n_1}X^{(1)}\Big(X_{\tau_+^{(2)}(t_4^{(1)})}^{(2)}-X^{(2)}_{\tau_-^{(2)}(t_3^{(1)})}\Big)=\Delta_4^{n_1}X^{(1)}\Big(\Delta_3^{n_2}X^{(2)}+\Delta_4^{n_2}X^{(2)}+\Delta_5^{n_2}X^{(2)}\Big)\end{align*}
as well. This is the illustration we prefer to analyze the jumps in $X^{(1)}$.
Starting with the symmetric illustration \eqref{varHY2}, the same terms appear, but rearranged in a different way and in several addends:
\begin{align*}&\Delta_3^{n_2}X^{(2)}\Big(X_{\tau_+^{(1)}(t_3^{(2)})}^{(1)}-X^{(1)}_{\tau_-^{(1)}(t_2^{(2)})}\Big)+\Delta_4^{n_2}X^{(2)}\Big(X_{\tau_+^{(1)}(t_4^{(2)})}^{(1)}-X^{(1)}_{\tau_-^{(1)}(t_3^{(2)})}\Big)+\Delta_5^{n_2}X^{(2)}\Big(X_{\tau_+^{(1)}(t_5^{(2)})}^{(1)}-X^{(1)}_{\tau_-^{(1)}(t_4^{(2)})}\Big)\\
&\quad=\Delta_3^{n_2}X^{(2)}\Big(\Delta_3^{n_1}X^{(1)}+\Delta_4^{n_1}X^{(1)}\Big)+\Delta_4^{n_2}X^{(2)}\Delta_4^{n_1}X^{(1)}+\Delta_5^{n_2}X^{(2)}\Big(\Delta_4^{n_1}X^{(1)}+\Delta_5^{n_1}X^{(1)}+\Delta_6^{n_1}X^{(1)}\Big).\end{align*}
This illustration simplifies treatment of jumps in $X^{(2)}$. Finally, from the refresh time illustration \eqref{varHY3} we find the same terms in the addends
\begin{align*}&\Big(\Delta_4^{-,n}X^{(1)}+\Delta_4^nX^{(1)}\Big)\Big(\Delta_4^{+,n}X^{(2)}+\Delta_4^nX^{(2)}\Big)+\Delta_3^nX^{(2)}\Big(\Delta_3^{+,n}X^{(1)}+\Delta_3^nX^{(1)}+\Delta_3^{-,n}X^{(1)}\Big)\\
&\quad=\Delta_4^{n_1}X^{(1)}\Big(\Delta_4^{n_2}X^{(2)}+\Delta_5^{n_2}X^{(2)}\Big)+\Delta_3^{n_2}X^{(2)}\Big(\Delta_3^{n_1}X^{(1)}+\Delta_4^{n_1}X^{(1)}\Big).\end{align*}
The effect of a jump is free from the particular illustration. It is convenient to consider the partition $[t_{i-1}^{(1)},t_i^{(1)}),$ $i=1,\ldots m_-^{(1)}(1),$ when we trace out jumps of $X^{(1)}$ and $[t_{j-1}^{(2)},t_j^{(2)}),$ $j=1,\ldots m_-^{(2)}(1),$ for jumps of $X^{(2)}$, while we use $[T_{k-1},T_k),$ $k=1,\ldots, M^n(1),$ for the continuous part. The main reason for the latter is that the estimation error can be written as sum of martingale differences when using refresh times; see \cite{asyn} for details.\qed
\end{exa}
\begin{figure}[t]
\centering
\fbox{\includegraphics[width=16cm]{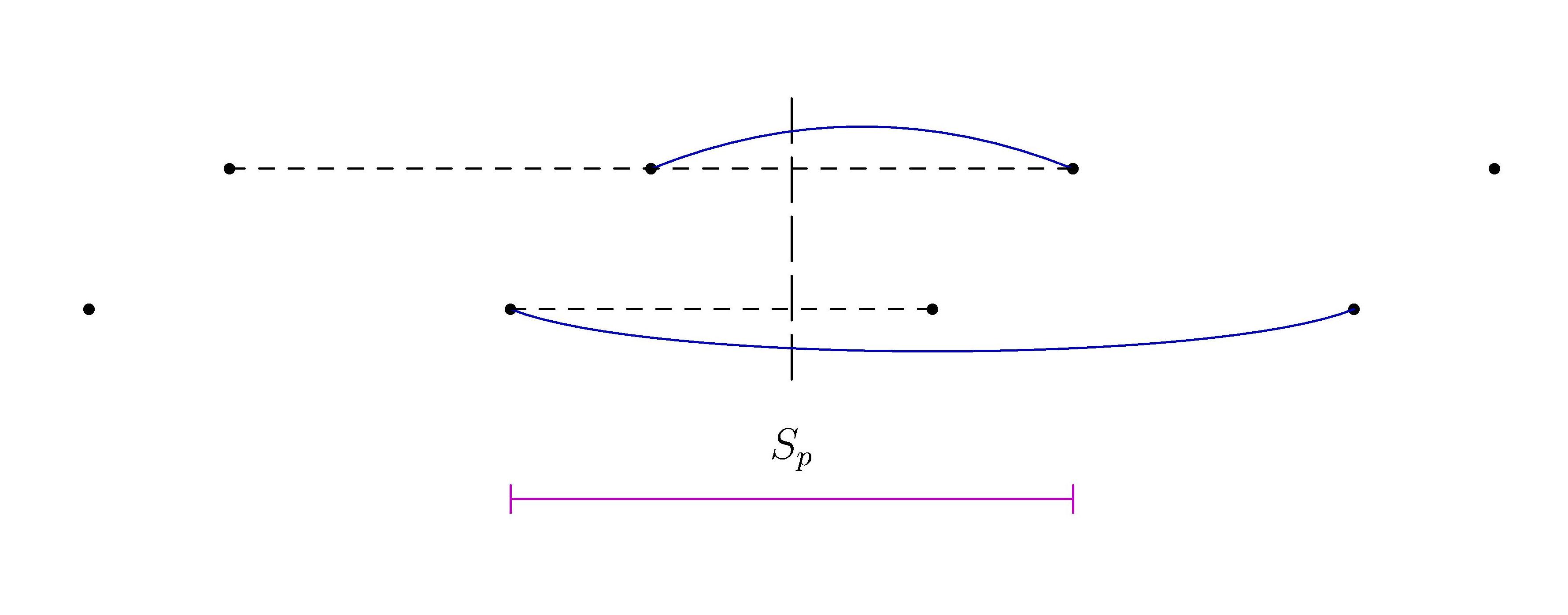}}
\caption{\label{fig2}A co-jump and intervals that determine the covariance structure.}
\end{figure}
\subsection{Asymptotic theory}
Say a co-jump occurs at time $S_p$. As can be seen from representation \eqref{varHY1}, the jump in $X^{(1)}$ is multiplied in the cross term with the increment of $X^{(2)}$ over the interpolated interval $[\tau_{--}^{(2,1)}(S_p),\tau_{++}^{(2,1)}(S_p)]$, and for the jump of $X^{(2)}$ symmetrically. The products are marked in Figure \ref{fig2} by the arcs and dashed segments, respectively. Idiosyncratic jumps are included in the general consideration by setting the jump in one component equal to zero. Similarly to the univariate case, the part due to jumps in the limiting variable is coming from a mixture of the particular jump of one process and the continuous increment of the other. Therefore, quantities like the length of $[\tau_{--}^{(2,1)}(S_p),\tau_{++}^{(2,1)}(S_p)]$ determine the contribution of one particular jump in the asymptotic variance. 

An intriguing effect arises by co-jumps in the multi-dimensional setting induced by the covariance of the two respective cross terms, since $d[X^{(1)},X^{(2)}]_s^c=\sigma_s^{(1)} \sigma_s^{(2)} \rho_s \,ds$. The covariance hinges on the intersection of both interpolated intervals in the two cross terms and results in an auxiliary condition such that the variance of the HY-estimator converges. In Figure \ref{fig2} this intersection is highlighted by the segment with bars. In any case, the following five intervals determine the variance of the HY-estimator by one particular co-jump at $s\in[0,1]$:
\begin{subequations}
\begin{align}\label{rl1}\Big(\mathcal{R}_n^1+\mathcal{L}_n^{1}\Big)(s)&=\max\Big(\tau_+^{(1)}(s),\tau_+^{(2)}(s)\Big)-\min\Big(\tau_-^{(1)}(s),\tau_-^{(2)}(s)\Big),\\
\label{r2}\mathcal{R}_n^2(s)&=\tau_{++}^{(1,2)}(s)-\max\Big(\tau_+^{(1)}(s),\tau_+^{(2)}(s)\Big),\\
\label{r3}\mathcal{R}_n^3(s)&=\tau_{++}^{(2,1)}(s)-\max\Big(\tau_+^{(1)}(s),\tau_+^{(2)}(s)\Big),\\
\label{l2}\mathcal{L}_n^2(s)&=\min\Big(\tau_-^{(1)}(s),\tau_-^{(2)}(s)\Big)-\tau_{--}^{(1,2)}(s),\\
\label{l3}\mathcal{L}_n^3(s)&=\min\Big(\tau_-^{(1)}(s),\tau_-^{(2)}(s)\Big)-\tau_{--}^{(2,1)}(s).
\end{align}
\end{subequations}
Either \eqref{r2} or \eqref{r3} is zero (both only in case of a synchronous observation), and the same is true for \eqref{l2} and \eqref{l3}. Yet, at each jump arrival $S_p$, we need to distinguish if $\mathcal{R}_n^2(S_p)>0$ or $\mathcal{R}_n^3(S_p)>0$. The segment with bars in Figure \ref{fig2} corresponds to $(\mathcal{R}_n^1+\mathcal{L}_n^{1})(S_p)$.

To derive a central limit theorem for the HY-estimator, already in the purely continuous case certain regularity conditions on the sequences of observation times are required; see \cite{hayayosh2011}. The analogous conditions using illustration \eqref{varHY3} from \cite{asyn} and additional conditions that ensure convergence of the variance in the presence of jumps are gathered in the next assumption: 
\begin{ass}\label{schememult}
Assume the $t_i^{(l)},l=1,2$, are stopping times such that $\E[\pi_n^q] = o(n^{-\alpha})$ for all $q \geq 1$ and any $0 < \alpha < q$. 
\begin{itemize}
	\item[(i)] Suppose that the functional sequences
\begin{subequations}
\begin{align}\label{qvt1}G_n(t)=n\sum_{T_k\le t}(\Delta_k^n)^2,\end{align}
\begin{align}\label{qvt2}F_n(t)=n\sum_{T_{k+1}\le t}\left((\Delta_k^n+\Delta_k^{-,n_2})\Delta_k^{+,n_1}+\Delta_k^{+,n_2}(\Delta_k^n+\Delta_k^{-,n_1})+\Delta_{k+1}^n(\Delta_{k+1}^{-,n_1}+\Delta_{k+1}^{-,n_2})\right),\end{align}
\begin{align}\label{qvt3}H_n(t)=n\sum_{T_{k+1}\le t}(\Delta_k^{-,n_1}\Delta_k^{+,n_1}+\Delta_k^{-,n_2}\Delta_k^{+,n_2}),\end{align}
\end{subequations}
converge, i.e.\ satisfy $G_n(t)\rightarrow G(t)$ pointwise for some continuously differentiable limiting function $G$, and analogously for $H_n,F_n$ with limits $H,F$.
\item[(ii)]
 Assume, for any $0 < t \leq 1$ and any $k \in \N$, convergence of
\begin{align} \label{reltermmult}
\int_{[0,t]^k} g(x_1, \ldots, x_k) \E\Big[\prod_{p=1}^k h_p\Big(\Big(n(\mathcal{R}_n^1+\mathcal{L}_n^1),n\mathcal{R}_n^2,n\mathcal{L}_n^2,n\mathcal{R}_n^3,n\mathcal{L}_n^3\Big)(x_p)\Big)\Big] dx_k \ldots dx_1,
\end{align}
with the expressions introduced in \eqref{rl1}--\eqref{l3}, to 
	\begin{align} \label{locexist2mult} 
	\int_{[0,t]^k} g(x_1, \ldots, x_k) \prod_{p=1}^k \int_{\R^5} h_p(y_1,y_2,y_3,y_4,y_5) \mathbf{\Phi}(x_p,d\mathbf{y}) dx_k \ldots dx_1
	\end{align}
for some family of probability measures $\mathbf{\Phi}$ on $[0,\infty)^5$ with finite first moment, holds true as $n\rightarrow\infty$, for all bounded continuous functions $g$ and $h_p$, $p=1, \ldots, k$.
\end{itemize}
\end{ass} 
We now introduce limiting variables for the central limit theorem in the two-dimensional case. They rely on the functions from Assumption \ref{schememult}. Denote
\begin{align*}\tilde {\mathcal{V}}_t= \int_0^tv_s\,dW'_s\end{align*}
with the variance process
\begin{align}\label{vprocess}v_s^2=G^{\prime}(s)\Big(\sigma_s^{(1)}\sigma_s^{(2)}\Big)^2(1+\rho_s^2)+F^{\prime}(s)\Big(\sigma_s^{(1)}\sigma_s^{(2)}\Big)^2+2H^{\prime}(s)\Big(\rho_s\sigma_s^{(1)}\sigma_s^{(2)}\Big)^2.\end{align}
The limit of the cross term is
\begin{align} \tilde{ \mathcal{Z}}_t&=\sum_{p:S_p\le t}\Delta X_{S_p}^{(1)}\sigma_{S_p}^{(2)}\Big(\sqrt{(\mathcal{R}^1+\mathcal{L}^1)(S_p)}U_p^{(1)}+\sqrt{\mathcal{R}^3(S_p)}U_p^{(3)}+\sqrt{\mathcal{L}^3(S_p)}Q_p^{(3)}\Big)\\
&\quad \notag +\Delta X_{S_p}^{(2)}\sigma_{S_p}^{(1)}\Big(\sqrt{(\mathcal{R}^1+\mathcal{L}^1)(S_p)}\Big(\rho_{S_p}U_p^{(1)}+\sqrt{1-\rho_{S_p}^2}Q_p^{(1)}\Big)+\sqrt{\mathcal{R}^2(S_p)}U_p^{(2)}+\sqrt{\mathcal{L}^2(S_p)}Q_p^{(2)}\Big),\end{align}
where $S_p$ denotes some enumeration of all times where at least one process jumps (so certain addends may become zero if a jump is idiosyncratic). Again, we need a second probability space $(\Omega',\mathcal{F}',\PP')$ on which mutually independent standard normal variables \((U_p^{(1)},U_p^{(2)},U_p^{(3)},\) \(Q_p^{(1)},Q_p^{(2)},Q_p^{(3)})\), $p \geq 1$, and  \((\mathcal{R}^1+\mathcal{L}^1,\mathcal{R}^2,\mathcal{L}^2,\mathcal{R}^3,\mathcal{L}^3)(x)\thicksim \mathbf{\Phi}(x,d{\bf{y}})\) for all $x \in [0,t]$ are defined. 

The second and third summand of \eqref{vprocess} give the limiting asymptotic variance of the error due to asynchronicity in the continuous part, i.e.\ the second term comes from the variance of the interpolation steps in the addends of \eqref{varHY3} and does not depend on the correlation whereas the third addend comes from the covariance between successive summands in \eqref{varHY3}. We refer to \cite{asyn} for further details and examples for the asymptotic theory concerning the continuous semimartingale part.
\begin{theo}\label{theo2}
Suppose we observe a two-dimensional It\^{o} semimartingale \eqref{multivX} whose characteristics fulfill the structural Assumption \ref{coeff2}. If the mesh of discrete observations tends to zero in both components, for each $0\le t\le 1$, the HY-estimator consistently estimates the quadratic covariation 
\begin{align}\widehat{\left[ X^{(1)}, X^{(2)}\right]}_t^{(HY),n}\stackrel{\PP}{\longrightarrow}[X^{(1)},X^{(2)}]_t\,.\end{align}
On Assumption \ref{schememult}, an $\mathcal F^0$-stable central limit theorem applies:
\begin{align}\sqrt{n}\widehat{\left[ X^{(1)}, X^{(2)}\right]}_t^{(HY),n}\tols \tilde {\mathcal{V}}_t+\tilde{\mathcal{ Z}}_t\label{cltmulti}\,.\end{align}
\end{theo}
\begin{exa} \rm Let us discuss a particular example again. By virtue of the symmetry,
\begin{align*}\mathbb{P}\Big(n\mathcal{L}_n^1(s)\le u\Big)=\begin{cases}1,&u=ns,\\ \mathbb{P}\Big(n\mathcal{R}_n^1(s-u)\le u\Big)&0<u<ns,\end{cases}\end{align*}
when observation times are modeled by renewal processes (if we have time-homogeneity), which means that under high-frequency asymptotics the distributions of backward and forward waiting times are asymptotically equal. Consequently, this is valid for time-homogenous Poisson sampling. Furthermore, by the strong Markov property both the $\mathcal{R}_n$ and $\mathcal{L}_n$, and the $\mathcal{R}_n^1$ and $(\mathcal{R}_n^2,\mathcal{R}_n^3)$ variables are independent. 

Precisely, suppose that $(N^n_l)_t$, $l=1,2,$ are independent Poisson processes with intensities $n\lambda_l$ for $\lambda_l > 0$ fixed and each $n$, and that the observations times are equivalent to the jump times of the respective processes. We find that constantly in time
\begin{align*}\mathbb{P}\Big(\mathcal{R}_n^1\le u_1\Big)= \Big(1-\exp(-n\lambda_1u_1)\Big)\Big(1-\exp(-n\lambda_2u_1)\Big),\end{align*}
since the two Poisson processes are independent. Note further that one of the components of $(\mathcal{R}_n^2,\mathcal{R}_n^3)$ is zero, depending on which process is observed first after time $\max\big(\tau_+^{(1)}(s),\tau_+^{(2)}(s)\big)=r$. The other one follows an exponential distribution then. Precisely, we have 
\begin{align*}\mathbb{P}\Big(\tau_+^{(1)}(r)\ge \tau_+^{(2)}(r)\Big)&=\int_0^{\infty}\int_0^y\lambda_1\exp(-n\lambda_1y)\lambda_2\exp(-n\lambda_2x)\,dx\,dy=\lambda_2(\lambda_1+\lambda_2)^{-1},\end{align*}
thus
\begin{align*}\mathbb{P}\Big(\mathcal{R}_n^2\leq u_2,\mathcal{R}_n^3\leq u_3 \Big)=\lambda_1(\lambda_1+\lambda_2)^{-1}\Big(1-\exp(-n\lambda_1u_2)\Big)+\lambda_2(\lambda_1+\lambda_2)^{-1}\Big(1-\exp(-n\lambda_2u_3)\Big),\end{align*}
from which we obtain
\begin{align*}\mathbb{P}\Big(\mathcal{R}_n^1\le u_1,\mathcal{R}_n^2\le u_2,\mathcal{R}_n^3\le u_3\Big)&=\Big(1-\exp(-n\lambda_1u_1)\Big)\Big(1-\exp(-n\lambda_2u_1)\Big)\\
&\quad\times \Big(1-\lambda_1(\lambda_1+\lambda_2)^{-1}\exp(-n\lambda_1u_2)-\lambda_2(\lambda_1+\lambda_2)^{-1}\exp(-n\lambda_2u_3)\Big).\end{align*}
This joint law enters the random limiting variance of the HY-estimator with jumps for Poisson sampling, that is 
\begin{align*}\mathbb{P}\Big(\mathcal{R}^1\le u_1,\mathcal{R}^2\le u_2,\mathcal{R}^3\le u_3\Big)&=\Big(1-\exp(-\lambda_1u_1)\Big)\Big(1-\exp(-\lambda_2u_1)\Big)\\
&\quad\times \Big(1-\lambda_1(\lambda_1+\lambda_2)^{-1}\exp(-\lambda_1u_2)-\lambda_2(\lambda_1+\lambda_2)^{-1}\exp(-\lambda_2u_3)\Big)\end{align*}
and the vector $(\mathcal{L}^1,\mathcal{L}^2,\mathcal{L}^3)$ follows the same distribution. 

For example, if $\lambda_1=\lambda_2=1$, we obtain \(4\Big(\sigma_{S_p}^{(1)}\Delta X^{(2)}_{S_p}\Big)^2+4\Big(\sigma_{S_p}^{(2)}\Delta X^{(1)}_{S_p}\Big)^2+6\rho_{S_p}\sigma_{S_p}^{(1)}\sigma_{S_p}^{(2)}\Delta X^{(1)}_{S_p}\Delta X^{(2)}_{S_p}\) at each jump time $S_p$ as the conditional expectation of the induced variance of the HY-estimator. In this case the covariation of time functions \eqref{qvt1}-\eqref{qvt3} are $G(t)=(14/9)t,$ $H(t)=(2/9)t,$ $F(t)=(10/9)t;$ for the latter compare \cite{asyn} and \cite{hayayosh2011}. \qed
\end{exa}
An estimator for the asymptotic variance of the HY-estimator facilitating a so-called feasible limit theorem to construct confidence intervals can be deduced similarly as in the continuous framework, pursued in \cite{hayayosh2011} by a kernel-type and \cite{asyn} by a histogram-type approach. See also \cite{veraart2010} for the univariate setting including jumps with regular observations. 

\section{Simulations}\label{sec:5}
In this section we inspect the HY-estimator's finite-sample accuracy and confirm our theoretical findings in a Monte Carlo simulation study. For this purpose, we implement a two-dimensional semimartingale model
with constant drift $b=(0.1,0.1)^{\top}$ and $\sigma^{(1)}=\sigma^{(2)}=1$ and a correlation parameter $\rho$. As a jump component we generate from a conditional Poisson process -- with a fixed number of jumps at times uniformly distributed on $[0,1]$ and jump heights 1, which allows for a simple tracking of the estimator's variance for different parameters and a comparison to the theoretical values. The non-synchronous observation times are randomly generated from two independent homogenous Poisson processes with expected time instants $1/30000$ for each (such that in average $m_-^{(1)}(1)=m_-^{(2)}(1)=30000$ observations on $[0,1]$). \\
In each iteration the observation scheme is newly generated. We run 500 Monte Carlo iterations for each configuration. From the simulated continuous component we evaluate the HY-estimator without jumps and from the whole process the HY estimator with jumps. We consider three simple setups:
\begin{itemize}
\item[{[Sc1]}] One co-jump (size one in both components) occurs.
\item[{[Sc2]}] One co-jump and one idiosyncratic jump in $X^{(1)}$ and one in $X^{(2)}$ (all size one) occur.
\item[{[Sc3]}] Only one idiosyncratic jump in $X^{(1)}$ and one in $X^{(2)}$ (all size one) occur.
\end{itemize}
The results are displayed in Table \ref{tab1}. First of all they confirm that the HY-estimator is eligible to consistently estimate the entire quadratic covariation in each scenario. In particular, the Monte Carlo averages show that the HY-estimator closely tracks the integrated covolatility in Scenario 3, robustly to the simulated large idiosyncratic jumps. We draw a comparison between the empirical finite-sample values and theoretical quantities obtained from the expectation of the random asymptotic variance and the convergence rate. Theoretical expected values are given in parentheses in Table \ref{tab1}. The variances of the continuous part and the variance due to the cross term by jumps and the continuous component are close to their theoretical counterparts. In Scenario 3 empirical variances are slightly larger than expected. The estimator's variance does not depend on the correlation here, since only idiosyncratic jumps occur. This is apparent also for the empirical figures. In Scenarios 1 and 2 one addend of the variance of the cross term increases linearly with the correlation what we witness as well for the finite-sample results.
\newpage
\begin{landscape}
\renewcommand{\arraystretch}{1.75}
\begin{table}
\vspace*{.5cm}
\caption{\label{tab1}Simulation results for different parameters.}
\vspace*{.5cm}
\begin{tabular}{|c|cccc|cc|cc|}
\hline
&\multicolumn{4}{|c|}{[Sc1]}&\multicolumn{2}{|c|}{[Sc2]}&\multicolumn{2}{|c|}{[Sc3]}\\
\hline
$\rho$& mean(HY)& mean(HY(C))& n$\cdot$ var(HY(J))& n$\cdot$ var(HY(C))& mean(HY)&n$\cdot$  var(HY(J))& mean(HY)& n$\cdot$ var(HY(J))\\
\hline
0&  1.000&0.000&7.91 (8.00)&3.87 (4.00)&1.002&15.95 (16.00)& 0.001& 9.03\\
0.1&1.101&0.100&8.49 (8.60)&4.03 (4.03)&1.100&16.02 (16.60)& 0.099& 8.72\\
0.2&1.200&0.200&8.99 (9.20)&4.05 (4.12)&1.199&17.45 (17.20)& 0.200& 8.47\\
0.3&1.300&0.300&9.56 (9.80)&4.40 (4.27)&1.299&17.83 (17.80)&0.297 &8.00\\
0.4&1.401&0.401&10.25 (10.40)&4.50 (4.48)&1.401& 18.43 (18.40)&0.399 &8.35 \\
0.5&1.500&0.499&11.01 (11.00)&4.52 (4.75)&1.500&19.11 (19.00)&0.501 &8.13 \\
0.6&1.599&0.599&11.83 (11.60)&4.78 (5.08)&1.599& 20.21 (19.60)&0.600 &8.54 \\
0.7&1.701&0.700&12.18 (12.20)&5.41 (5.47)&1.701& 20.26 (20.20)& 0.700& 9.08\\
0.8&1.801&0.801&12.74 (12.80)&6.07 (5.92)&1.799& 20.35 (20.80)& 0.803 & 8.44\\
0.9&1.900&0.900&13.14 (13.40)&6.49 (6.43)&1.899&21.08 (21.40)&0.901 &8.21 \\
1.0&2.001&1.001&13.36 (14.00)&6.98 (7.00)&1.999&21.90 (22.00)&0.999 &8.57 \\
\hline
\end{tabular}
\begin{quote}{Note. $\E[m_-^{(1)}(1)]=\E[m_-^{(2)}(1)]=30000$; HY(C) refers to the estimator calculated from the simulated continuous part only; var(HY(J)) is the difference of the total Monte Carlo variance and var(HY(C)); in parentheses theoretical expectations, in [Sc3] this is constantly 8.}\end{quote}
\end{table}
\end{landscape}
\newpage
\section*{Appendix}
\appendix
\section{Preliminaries and some notation \label{app1}}
\setcounter{equation}{0}
\renewcommand{\theequation}{\thesection.\arabic{equation}} 
Throughout the proof $K$ and $K_q$ denote generic constants, the latter dependent on $q$. On the compact time span $[0,1]$, we can reinforce the structural Assumption \ref{coeff} replacing local boundedness by uniform boundedness which precisely means that $b$ and $\sigma$ as well as the jumps of $X$ may be assumed to be bounded by $K$. Such a standard procedure is provided in Section 3.6.3.\ in \cite{jacodlecture}, among others.

The following notation is analogously introduced for the one-dimensional and two-dimensional setting. For any integer $q$ we define the auxiliary (drift) process
\[
b_s^q = b_s - \int (\kappa(\delta(s,z)) - \delta(s,z) 1_{\{\gamma(z) \leq 1/q\}}) \lambda (dz),
\] 
which satisfies $\|b_s^q\| \leq K_q$. We consider discretized versions $b(q,r)_s$ and $\sigma(r)_s$ of $b^q$ and $\sigma$ for some integer $r$.
We set
\[b(q,r)_s=b_{(k-1)/2^r}^q 1_{[(k-1)/2^r,k/2^r)},~k=1,\ldots,2^r,\]
locally constant on intervals $[(k-1)/2^r,k/2^r)$ and analogously for $\sigma(r)$. 
For any fixed pair $(q,r)$, we define
\begin{align*}
&B(q,r)_t  = \int_0^t b(q,r)_s\, ds, \quad B'(q,r)_t =  \int_0^t (b^q_s - b(q,r)_s)\, ds, \\
&C(r)_t = \int_0^t \sigma(r)_s\, dW_s, \quad C'(r)_t = \int_0^t (\sigma_s - \sigma(r)_s) \,dW_s, \\ 
&N(q)_t = \int_0^t \int \delta(s,z) 1_{\{\gamma(z) > 1/q\}} \mu(ds,dz), \quad M(q)_t = \int_0^t \int \delta(s,z) 1_{\{\gamma(z) \leq 1/q\}} (\mu-\nu)(ds,dz). 
\end{align*}
Up to $X_0$, which does not matter in terms of increments anyway, $X_t$ can be written as the sum of the six quantities above. Therefore each increment $\Delta_i^n X$ is the sum of six respective increments as well, and using the multinomial formula we see that $RV_t^n$ (or the HY-estimator) becomes a sum of 21 addends. For the quadratic variation (for dimension one), we have 
\begin{align*}
[X,X]_t = \int_0^t \sigma(r)_s^2 \,ds + \int_0^t (\sigma_s-\sigma(r)_s)^2 \,ds + 2 \int_0^t \sigma(r)_s (\sigma_s-\sigma(r)_s) ds + [N(q),N(q)]_t + [M(q),M(q)]_t
\end{align*}
and analogously for $[X^{(1)},X^{(2)}]$ in the two-dimensional setup.
 Next, we state some standard estimates for the terms from the above decomposition of $X$ which will be used frequently in the analysis below. See e.g.\ Section 4.1\ in \cite{jacod2008}:
\begin{subequations}
 \begin{align}\label{e1}\forall p\ge 1,s,t\ge 0:~\E\left[\|C_{s+t}(r)-C_s(r)\|^p\Big|\mathcal{F}_s\right]\le K_p t^{\nicefrac{p}{2}},\end{align}
 %\begin{align}\label{e1}\forall s,t\ge 0:~\E\left[\|C_{s+t}(r)-C_s(r)\|^2\Big|\mathcal{F}_s\right]\le K t\,,\end{align}
 %\begin{align}\notag \forall p\ge 1,s,t\ge 0:~\E\left[\|M(q)_{s+t}-M(q)_s\|^p\Big|\mathcal{F}_s\right]&\le K_{p}\,\E\left[\left(\int_s^{(s+t)}\int_{\{z\in\R^d|\gamma(z)<q^{-1}\}}\gamma^2(z)\nu(ds,dz)\right)^{\frac{p}{2}}\right]\\ \label{e2}&  \le K_p t^{(\frac{p}{2}\wedge 1)}e_q^{(\frac{p}{2}\wedge 1)}\,,\end{align}
 \begin{align}\notag \forall s,t\ge 0:~\E\left[\|M(q)_{s+t}-M(q)_s\|^2\Big|\mathcal{F}_s\right]&\le K\,\E\left[\left(\int_s^{(s+t)}\int_{\{z\in\R^d|\gamma(z)<q^{-1}\}}\gamma^2(z)\nu(ds,dz)\right)\right]\\ \label{e2}&  \le K te_q,\end{align}
%\begin{align}\notag \forall p\ge 1,s,t\ge 0:~&\E\left[\|C'_{s+t}(r)-C'_s(r)\|^p\Big|\mathcal{F}_s\right]
 %\\ &\le K_p \,\E\left[\left(\int_s^{s+t}\|\sigma(r)_{\tau}-\sigma(r)_s\|^2\,d\tau\right)^{\frac{p}{2}}\Big|\mathcal{F}_s\right]
 %\label{e3} =o \Big( t^{p/2}\Big)\,,\end{align}
 \begin{align}\notag \forall s,t\in [(k-1)2^{-r},k2^{-r})~ \mbox{for some}~k:~\E\left[\|C'_{s+t}(r)-C'_s(r)\|^2\Big|\mathcal{F}_s\right]
  &\le K \,\E\left[\left(\int_s^{s+t}\|\sigma_{\tau}-\sigma_s\|^2\,d\tau\right)\Big|\mathcal{F}_s\right]\\
 \label{e3} &\le K t\sup_{\tau\in[s,t+s]}\|\sigma_{\tau}-\sigma_s\|^2,\end{align}
  %\begin{align}\label{e4}\forall s,t\ge 0:~\E\left[\|N(q)_{s+t}-N(q)_s\|\Big|\mathcal{F}_s\right]\le q\,t\, \int \gamma^2(z)\lambda(dz)\le K_q \,t\,,\end{align}
    \begin{align}\label{e4}\forall s,t\ge 0:~\E\left[\|N(q)_{s+t}-N(q)_s\|\Big|\mathcal{F}_s\right]\le q\,t\, \int \gamma^2(z)\lambda(dz)\le K_q \,t,\end{align}
where $e_q=\int_{\{z\in\R^d|\gamma(z)<q^{-1}\}}\gamma^2(z)\lambda(dz)$. From Assumption \ref{coeff} we may conclude that $\int\gamma^2(z)\lambda(dz)$ is bounded, so by Lebesgue's theorem we have
\begin{align}\label{e5}e_q \rightarrow 0~~\mbox{as}~~q\rightarrow\infty.\end{align}
\end{subequations} 
 
\section{Proof of Theorem \ref{theo1} \label{app2}}
\setcounter{equation}{0}
\renewcommand{\theequation}{\thesection.\arabic{equation}} 
We decompose the left hand side of (\ref{cltuni}) with the terms introduced in Appendix \ref{app1}. The only addends responsible for the limiting variance are 
	\begin{align} \label{step1}
	n^{1/2} \Big[\Big(\sum_{i=1}^{\lfloor nt \rfloor} (\Delta_i^n C(r))^2 -  \int_0^t \sigma(r)_s^2 ds \Big) + 2 \sum_{i=1}^{\lfloor nt \rfloor} \Delta_i^n C(r)\Delta_i^n N(q) \Big]. 
	\end{align}
For the other terms converging to the quadratic variation the approximation errors 
	\begin{align}
	& n^{1/2} \Big[ \Big(\sum_{i=1}^{\lfloor nt \rfloor} (\Delta_i^n C'(r))^2 -  \int_0^t (\sigma_s - \sigma(r)_s)^2 ds\Big) + \Big(\sum_{i=1}^{\lfloor nt \rfloor} (\Delta_i^n N(q))^2 - [N(q),N(q)]_t \Big) \label{step2}\\ \nonumber
	&+2 \Big(\sum_{i=1}^{\lfloor nt \rfloor} \Delta_i^n C(r) \Delta_i^n C'(r) -  \int_0^t \sigma(r)_s(\sigma_s - \sigma(r)_s) ds\Big) + \Big(\sum_{i=1}^{\lfloor nt \rfloor} (\Delta_i^n M(q))^2 - [M(q),M(q)]_t \Big) \Big],
	\end{align}
are proved to be asymptotically negligible.
	The other remainder terms will be shown to be small as well. We state an overview which terms are treated jointly at this point: 
The pure drift parts are 
\begin{align} \label{step3}
n^{1/2} \sum_{i=1}^{\lfloor nt \rfloor} (\Delta_i^n B(q,r) + \Delta_i^n B'(q,r))^2,
\end{align}
and we also treat
\begin{align} \label{step4}
2 n^{1/2} \sum_{i=1}^{\lfloor nt \rfloor} (\Delta_i^n B(q,r) + \Delta_i^n B'(q,r))(\Delta_i^n C(r) + \Delta_i^n N(q))
\end{align}
together. The mixed martingale part is 
\begin{align} \label{step5}
2 n^{1/2} \sum_{i=1}^{\lfloor nt \rfloor} \Delta_i^n M(q) (\Delta_i^n C(r) + \Delta_i^n C'(r)). 
\end{align}
The remainder terms involving $C'(r)$ are now
\begin{align} \label{step6}
2 n^{1/2} \sum_{i=1}^{\lfloor nt \rfloor} \Delta_i^n C'(r) (\Delta_i^n B(q,r) + \Delta_i^n B'(q,r) + \Delta_i^n N(q)),
\end{align} 
and finally the remainder terms involving $M(q)$ become
\begin{align} \label{step7}
2 n^{1/2} \sum_{i=1}^{\lfloor nt \rfloor} \Delta_i^n M(q) (\Delta_i^n B(q,r) + \Delta_i^n B'(q,r) + \Delta_i^n N(q)).
\end{align}
Asymptotics will always work in the sense that we let $n \to \infty$ first, then the auxiliary $r \to \infty$ and finally $q \to \infty$. 
The proof of Theorem \ref{theo1} is divided in two parts which are given in the following propositions:
\begin{prop}\label{propclt}On the assumptions of Theorem \ref{theo1}:
\begin{align}\label{step1prop}(\ref{step1})\tols V_t+Z_t\,.\end{align}
\end{prop}
%Formally, this means that we prove
\begin{prop}\label{propremainder} On the assumptions of Theorem \ref{theo1}:
\begin{align*}
\lim_{q \to \infty} \limsup_{r \to \infty} \limsup_{n \to \infty} \PP(|(\ref{step2}) + (\ref{step3}) + (\ref{step4}) + (\ref{step5}) + (\ref{step6}) + (\ref{step7})| > \epsilon) = 0
\end{align*}
for any $\epsilon > 0$.
\end{prop}
In this section, we show the claim of the latter proposition only. We will prove an analogous claim for Theorem \ref{theo1b} later, which includes Proposition \ref{propclt} as a special case. To establish Proposition \ref{propremainder} we show that it holds for each of the remainder terms (\ref{step2}) to (\ref{step7}) separately, which implies our claim. Note by Markov inequality that it is sufficient to obtain bounds for moments of the respective terms.

\begin{proof}[Proof of Proposition \ref{propremainder}]\hfill\\
\noindent
$\bullet ~~~\lim_{q \to \infty} \limsup_{r \to \infty} \limsup_{n \to \infty} \PP\left(|\eqref{step3}|>\epsilon \right)=0~\text{for any}~\epsilon>0$:\\
Since we have $|\Delta_i^n B(q,r)| \leq K_q /n$ and also for $\Delta_i^n B'(q,r)$
\[n^{1/2} \sum_{i=1}^{\lfloor nt \rfloor} (\Delta_i^n B(q,r) + \Delta_i^n B'(q,r))^2 \stackrel{\PP}{\longrightarrow} 0~\text{for any fixed $q$ and $r$.}\]
$\bullet ~~~\lim_{q \to \infty} \limsup_{r \to \infty} \limsup_{n \to \infty} \PP\left(|\eqref{step4}|>\epsilon \right)=0~\text{for any}~\epsilon>0$:
\[n^{1/2} \sum_{i=1}^{\lfloor nt \rfloor} \Delta_i^n B(q,r) \Delta_i^n C(r) = n^{-1/2} \sum_{k=1}^{\lfloor 2^rt\rfloor} b^q_{(k-1)2^{-r}} (C(r)_{k2^{-r}} - C(r)_{(k-1)2^{-r}}) + R(q,r)_t^n,
\]
where the error due to increments over intervals with $(i-1)/n \leq (k-1)2^{-r} < i/n$ and boundary effects is denoted by $R(q,r)_t^n$ and satisfies $\E[|R(q,r)_t^n|] \leq K_r n^{-1/2}$, thus becomes small. The first term on the right hand side above is a sum of martingale differences. Therefore
\begin{align*}n^{-1} \E \left(\sum_{k=1}^{\lfloor 2^rt\rfloor} b^q_{(k-1)2^{-r}} (C(r)_{k2^{-r}} - C(r)_{(k-1)2^{-r}})\right)^2 &\leq K/n \sum_{k=1}^{\lfloor 2^rt\rfloor}  \E \left(C(r)_{k2^{-r}} - C(r)_{(k-1)2^{-r}}\right)^2 \\ & \leq K /n \to 0.
\end{align*} 
For the treatment of the analogous term involving $\Delta_i^n B'(q,r)$ let $w(f,h) = \sup\{|f(s) - f(t)|; |s-t| \leq h\}$ denote the modulus of continuity of a function $f$. Several applications of Cauchy-Schwarz inequality and It\^o isometry  give
\begin{align*}
n^{1/2} \sum_{i=1}^{\lfloor nt \rfloor} \E |\Delta_i^n B'(q,r) \Delta_i^n C(r)| \leq n^{-1/2} \sum_{i=1}^{\lfloor nt \rfloor} \E |w(b^q,2^{-r})\Delta_i^n C(r)| \leq K \E[|w(b^q,2^{-r})|^2]^{1/2}.
\end{align*}
By construction, $b^q$ is bounded and continuous. Therefore, $w(b^q,2^{-r})$ converges to zero for each $\omega$ as $r \to \infty$, and the entire term becomes small due to Lebesgue's theorem. The claim for (\ref{step4}) now follows using the bounds on the drift and \eqref{e4}. \\
$\bullet ~~~\lim_{q \to \infty} \limsup_{r \to \infty} \limsup_{n \to \infty} \PP\left(|\eqref{step5}|>\epsilon \right)=0~\text{for any}~\epsilon>0$:\\
Integration by parts formula gives
\begin{align} \label{XM} \nonumber
\Delta_i^n C(r) \Delta_i^n M(q) =& \int_{(i-1)/n}^{i /n} (C(r)_{s-}-C(r)_{(i-1)/n})\, dM(q)_s \\ &+ \int_{(i-1)/n}^{i /n} (M(q)_{s}-M(q)_{(i-1)/n}) \,dC(r)_s,
\end{align}
with expectation zero. Applying Burkholder-Davis-Gundy inequality yields
\allowdisplaybreaks[3]{\begin{align} \nonumber
&\E \Big(\int_{(i-1)/n}^{i /n} (C(r)_{s-}-C(r)_{(i-1)/n})\, dM(q)_s\Big)^2  \leq K\E\int_{(i-1)/n}^{i /n} (C(r)_{s-}-C(r)_{(i-1)/n})^2 \,d[M(q),M(q)]_s \nonumber \\ \nonumber \leq& K\E\int_{(i-1)/n}^{i /n} \int (C(r)_{s-}-C(r)_{(i-1)/n})^2 \delta^2(s,z) 1_{\{\gamma(z) \leq 1/q\}} \lambda(dz) \,ds \\ \label{marting} \leq& K e_q \Big(\int_{(i-1)/n}^{i /n} \E (C(r)_{s-}-C(r)_{(i-1)/n})^2 ds \Big) \leq K e_q /n^2\,.
\end{align}}
%with $e_q = \int \gamma^2(z) 1_{\{\gamma(z) \leq 1/q\}} \lambda(dz)$. 
Similarly, we obtain
\[\E \Big(\int_{(i-1)/n}^{i /n} (M(q)_{s}-M(q)_{(i-1)/n})\, dC(r)_s\Big)^2 \leq K e_q /n^2. \]
and altogether 
\[n \E\Big(\sum_{i=1}^{\lfloor nt \rfloor} \Delta_i^n M(q) \Delta_i^n C(r)\Big)^2 \leq n  \sum_{i=1}^{\lfloor nt \rfloor} \E\Big(\Delta_i^n M(q) \Delta_i^n C(r)\Big)^2 \leq K e_q,\]
and the same bound holds for the term involving $\Delta_i^n C'(r)$. Our claim now follows by virtue of \eqref{e5}.
$\bullet ~~~\lim_{q \to \infty} \limsup_{r \to \infty} \limsup_{n \to \infty} \PP\left(|\eqref{step6}|>\epsilon \right)=0~\text{for any}~\epsilon>0$:
\begin{align*}& \Delta_i^n B(q,r) + \Delta_i^n B'(q,r) + \Delta_i^n N(q) = \int_{(i-1)/n}^{i/n} b_s^q ds + \int_{(i-1)/n}^{i/n} \int \delta(s,z) 1_{\{\gamma(z) > 1/q\}} \mu(ds,dz) 
%\\ =& \int_{(i-1)/n}^{i/n} \Big(b_s^q + \int \delta(s,z) 1_{\{\gamma(z) > 1/q\}} \lambda(dz)\Big) ds + \int_{(i-1)/n}^{i/n}\int \delta(s,z) 1_{\{\gamma(z) > 1/q\}} (\mu-\nu)(ds,dz) 
\\ =& \underbrace{\int_{(i-1)/n}^{i/n} \Big(b_s + \int \kappa'(\delta(s,z)) \lambda(dz)\Big) ds}_{=\Delta_i^n\widetilde B} +\underbrace{ \int_{(i-1)/n}^{i/n} \int \delta(s,z) 1_{\{\gamma(z) > 1/q\}} (\mu-\nu)(ds,dz)}_{=\Delta_i^n\widetilde M(q)} . 
\end{align*} 
By two applications of Cauchy-Schwarz inequality and It\^o isometry, we conclude
\begin{align*}n^{1/2} \sum_{i=1}^{\lfloor nt \rfloor} \E \Big[\Big|\Delta_i^n C'(r) \Delta_i^n\widetilde B\Big| \Big] \leq K n^{-1/2}\sum_{i=1}^{\lfloor nt \rfloor} \E [|\Delta_i^n C'(r)|^2]^{1/2} \leq K \E[|w(\sigma^2,2^{-r})|^2]^{1/2}.\end{align*}
Convergence to zero as $r \to \infty$ can be deduced from Lebesgue's theorem again. On the other hand, we have a similar decomposition for $\Delta_i^n C'(r) \Delta_i^n\widetilde M(q)$ as in (\ref{XM}). The arguments from (\ref{marting}) yield
\[
n \E \Big( \sum_{i=1}^{\lfloor nt \rfloor}\int_{(i-1)/n}^{i /n} (C'(r)_{s-}-C'(r)_{(i-1)/n})\, d\widetilde M(q)_s\Big)^2 \leq K \E[|w(\sigma^2,2^{-r})|^2],
\]
so it remains to focus on 
\[
n^{1/2} \sum_{i=1}^{\lfloor nt \rfloor}\int_{(i-1)/n}^{i /n} (\widetilde M(q)_{s-}-\widetilde M(q)_{(i-1)/n}) (\sigma_s - \sigma(r)_s) \,dW_s.
\]
We have that $\Delta_i^n\widetilde  M(q) = \Delta_i^n N(q)-\int_{(i-1)/n}^{i/n}\int \delta(s,z) 1_{\{\gamma(z)>q^{-1}\}}\lambda(dz) ds$, where the absolute value of the second addend is bounded by $K_q/n$. Since $\sigma_s - \sigma(r)_s \to 0$ pointwise for $r \to \infty$, all we need to discuss is 
\[n^{1/2} \sum_{i=1}^{\lfloor nt \rfloor}\int_{(i-1)/n}^{i /n} (N(q)_{s-}-N(q)_{(i-1)/n}) (\sigma_s - \sigma(r)_s)\, dW_s.\]
Let $S_1, S_2, \ldots$ be a sequence of stopping times exhausting the jumps of $N(q)$. Then 
\begin{align} \label{summe}
& n^{1/2} \sum_{i=1}^{\lfloor nt \rfloor} \E \Big| \int_{(i-1)/n}^{i /n} (N(q)_{s-}-N(q)_{(i-1)/n}) (\sigma_s - \sigma(r)_s)\, dW_s  \Big| \\ \nonumber \leq&\, n^{1/2} \sum_{p\geq 1} \sum_{i=1}^{\lfloor nt \rfloor} \E \left|1_{\{(i-1)/n < S_p \leq i/n\}} |N(q)_{S_p}-N(q)_{S_p-}|   \int_{S_p}^{\left\lceil nS_p\right\rceil/n} (\sigma_s - \sigma(r)_s)\, dW_s  \right|.
\end{align}
Using the strong Markov property of Brownian motion we obtain
\begin{align*}
\E \Big[\Big|\int_{S_p}^{\left\lceil nS_p\right\rceil/n} (\sigma_s - \sigma(r)_s) \,dW_s \Big| \Big| \f_{S_p}\Big] \leq& \Big( \int_{S_p}^{\left\lceil nS_p\right\rceil/n} \E[(\sigma_s - \sigma(r)_s)^2 | \f_{S_p}] ds \Big)^{1/2} \\ \leq& K n^{-1/2} \E[w(\sigma,2^{-r})^2 | \f_{S_p}]^{1/2}.
\end{align*}
Therefore, using Cauchy-Schwarz inequality several times, (\ref{summe}) can be bounded by a constant times
\begin{align} \nonumber
&\E \sum_{p\geq 1} \sum_{i=1}^{\lfloor nt \rfloor}  \left(1_{\{(i-1)/n < S_p \leq i/n\}} 1_{\{S_p \leq 1\}} \E[w(\sigma,2^{-r})^2 | \f_{S_p}]^{1/2} \right) \\ \leq&\,
\E\Big[\sum_{p \geq 1} \Big(\sum_{i=1}^{\lfloor nt \rfloor} 1_{\{(i-1)/n < S_p \leq i/n\}} \Big)^2 \Big]^{1/2} \E\Big[\sum_{p \geq 1} 1_{\{S_p \leq 1\}} \E[w(\sigma,2^{-r})^2 | \f_{S_p}] \Big]^{1/2}. \label{anders}
\end{align}
The first factor is bounded by $\E[\Pi(q)]^{1/2}$, where $\Pi(q) = \int_0^1 \int 1_{\{\gamma(z) > 1/q\}} \mu(ds,dz)$ denotes the number of large jumps on $[0,1]$. It holds that
\begin{align*}
\E[\left(\Pi(q)\right)^2] \leq& K \E\Big[ \Big(\int_0^1 \int 1_{\{\gamma(z) > 1/q\}} (\mu-\nu)(ds,dz) \Big)^2 + \Big(\int_0^1 \int 1_{\{\gamma(z) > 1/q\}} \lambda(dz) ds \Big)^2 \Big] \\ \leq& K \Big(\int_0^1 \int 1_{\{\gamma(z) > 1/q\}} \lambda(dz) ds + \Big(\int_0^1 \int 1_{\{\gamma(z) > 1/q\}} \lambda(dz) ds \Big)^2 \Big) \leq K_q
\end{align*}
by the integrability assumption on $\gamma^2$. It remains to focus on the second factor, for which
\begin{align*}
\sum_{p\geq 1} \E[1_{\{ S_p \leq 1\}} \E[w(\sigma,2^{-r})^2 | \f_{S_p}]^{1/2}] \leq& \sum_{p\geq 1} \E[1_{\{ S_p \leq 1\}}]^{3/4} \E(\E[w(\sigma,2^{-r})^2 | \f_{S_p}]^{2})^{1/4} \\ \leq& \E[w(\sigma,2^{-r})^4]^{1/4} \sum_{p\geq 1} \PP(S_p \leq 1)^{3/4}.
\end{align*}
using Hölder and Jensen inequality, respectively. Finally, 
\[
\sum_{p\geq 1} \PP(S_p \leq 1)^{3/4} = \sum_{p\geq 1} \PP(\Pi(q) > p)^{3/4} \leq (\E[(\Pi(q))^2])^{3/4}  \sum_{p\geq 1} p^{-3/2} \leq K_q.
\]
Continuity of $\sigma$ gives the claim again. \\
$\bullet ~~~\lim_{q \to \infty} \limsup_{r \to \infty} \limsup_{n \to \infty} \PP\left(|\eqref{step7}|>\epsilon \right)=0~\text{for any}~\epsilon>0$:\\
With the previous notation we have
\[
n^{1/2} \sum_{i=1}^{\lfloor nt \rfloor} \E \Big[\Big|\Delta_i^n M(q) \Delta_i^n \widetilde B\Big| \Big] \leq K n^{-1/2} \sum_{i=1}^{\lfloor nt \rfloor} \E [(\Delta_i^n M(q))^2]^{1/2} \leq K e_q^{1/2} \to 0
\]
as $q \to \infty$. The product $M(q)\widetilde M(q)$ is a martingale (no common jumps) and we may proceed as in the proof for (\ref{step5}). In particular
\begin{align} \nonumber
&\E \Big(\int_{(i-1)/n}^{i /n} \hspace*{-.1cm} (M(q)_{s-}-M(q)_{(i-1)/n})\, d\widetilde M(q)_s\Big)^2 \hspace*{-.1cm} \leq K\hspace*{-.05cm}\int_{(i-1)/n}^{i /n} \hspace*{-.1cm} \E(M(q)_{s-}-M(q)_{(i-1)/n})^2 \Big(\int \gamma^2(z) \lambda(dz)\Big) \,ds \nonumber \\ \leq& K\int_{(i-1)/n}^{i /n} \int_{(i-1)/n}^s \Big(\int \gamma^2(z) 1_{\{\gamma(z) \leq 1/q\}} \lambda(dz) \Big) \Big(\int \gamma^2(z) \lambda(dz)\Big) du ds \leq e_q K /n^2 . \label{twomartin}
\end{align}
Changing the roles of $M(q)$ and $\widetilde M(q)$ then gives the result. \\
$\bullet ~~~\lim_{q \to \infty} \limsup_{r \to \infty} \limsup_{n \to \infty} \PP\left(|\eqref{step2}|>\epsilon \right)=0~\text{for any}~\epsilon>0$:\\
We may replace $[X,X]_t$ with $[X,X]_{\lfloor nt \rfloor /n}$, since 
\[n^{1/2} \E([X,X]_t-[X,X]_{\lfloor nt \rfloor /n}) \leq K n^{-1/2}\,.\]
Now, let $\Omega(q,r,n,t)$ denote the set on which each interval $ [i/n,(i+1)/n]$ contains either one or no jump and no jump occurs on those intervals with $(i-1)/n \leq (k-1)2^{-r} < k2^{-r} \leq i/n$, $k=1, \ldots, 2^r$, and also not on $[\lfloor nt \rfloor /n,t]$.
%\begin{align}\label{omega}\Omega(q,r,n,t)=\Big\{\omega\in\Omega|\sup_{i\in{\{1,\ldots,n\}}/I}\left|\{S_p\in[i/n,(i+1)/n]||X_{S_p}-X_{S_p-}|>0\}\right|=1 ~\wedge  \\
% ~\notag\sup_{i\in I}\left|\{S_p\in[i/n,(i+1)/n]|X_{S_p}-X_{S_p-}>0\}\right|=0\Big\}\end{align}
%with $I=\{i|i>(k-1)2^{-r}/n^{-1}\ge i-1\}\cup \{\lfloor nt\rfloor\}$.
%For fixed $q, r, t$, $\Omega(q,r,n,t)\rightarrow \Omega$ as $n \to \infty$ which is why we may assume $\omega$ to live on this set. 
This is helpful, since 
\[n^{1/2} \Big(\sum_{i=1}^{\lfloor nt \rfloor} (\Delta_i^n N(q))^2 - [N(q),N(q)]_t \Big) = 0\]
identically on $\Omega(q,r,n,t)$. For fixed $q, r, t$, we have $\Omega(q,r,n,t)\rightarrow \Omega$ as $n \to \infty$ which is why we may assume $\omega$ to live on this set. (Note that we do not need the condition involving $r$ at this point, but the previously introduced set will be used several times later again.) Proving that
\[n^{1/2} \Big(\sum_{i=1}^{\lfloor nt \rfloor} (\Delta_i^n M(q))^2 - [M(q),M(q)]_t \Big) = 2 n^{1/2} \sum_{i=1}^{\lfloor nt \rfloor} \int_{(i-1)/n}^{i /n} (M(q)_{s-}-M(q)_{(i-1)/n})\, dM(q)_s\]
becomes small works similarly to (\ref{twomartin}). For the remaining terms in (\ref{step2}) let us exemplarily discuss 
\[
n^{1/2}\Big(\sum_{i=1}^{\lfloor nt \rfloor} (\Delta_i^n C'(r))^2 -  \int_0^t (\sigma_s - \sigma(r)_s)^2 ds\Big) = 2 n^{1/2} \sum_{i=1}^{\lfloor nt \rfloor} \int_{(i-1)/n}^{i /n} (C'(r)_{s}-C'(r)_{(i-1)/n})\, dC'(r)_s
\]
which is again a sum of martingale increments. Several applications of Burkholder-Davis-Gundy inequality and Cauchy-Schwarz inequality prove that the expectation of its square is bounded by $\E[w(\sigma,2^{-r})^4]^{1/2}$ which converges to zero as $r\rightarrow \infty$. A similar argument applies to the cross term involving $C(r)$ and $C'(r)$. This completes the proof of Proposition \ref{propremainder}.
\end{proof}

\section{Proof of Theorem \ref{theo1b} \label{app2b}}
\setcounter{equation}{0}
\renewcommand{\theequation}{\thesection.\arabic{equation}} 
We have the same decomposition for $n^{1/2} (RV_t^n - [X,X]_t)$ as in Section \ref{app2}, where the only difference is that sums run to $m_-^n(t)$ rather than $\lfloor nt \rfloor$. Therefore, we have to prove similar claims as Proposition \ref{propclt} and Proposition \ref{propremainder}, and the latter works in pretty much the same way as before, up to using Assumption \ref{scheme} in some places: To be precise, all claims follow by simply using \eqref{Hexist} where necessary, apart from the first term in \eqref{anders} which becomes the square root of
\begin{align*}
\E\Big[\sum_{p\geq 1} \Big(\sum_{i=1}^{m_-^n(t)} 1_{\{t_{i-1}^n \leq S_p \leq t_i^n\}} n^{1/2} (t_i^n - t_{i-1}^n)^{1/2}\Big)^2\Big] &\leq \E\Big[\sum_{p= 1}^{\Pi(q)}  n (\tau_+^n(S_p) - \tau_-^n(S_p))\Big] \\ &= \E\Big(\sum_{p= 1}^{\Pi(q)}  \E[n (\tau_+^n(S_p) - \tau_-^n(S_p))] \Big)
\end{align*}
in this context, where we have used the Wald identity to obtain the latter equality. If we order the stopping times by the size of the corresponding jumps, i.e.\ $S_1$ denotes the time of the largest jump of $N(q)$, $S_2$ is the time of the second largest jump, and so on, then each $S_p$ is uniformly distributed. Therefore $\E[n (\tau_+^n(S_p) - \tau_-^n(S_p))] \leq K$ using Assumption \ref{scheme} (ii), and we obtain the same bound as in \eqref{anders}.

For this reason, we focus in this proof on

\begin{prop}\label{propcltb} On the assumptions of Theorem \ref{theo1b}, we have the $\mathcal F^0$-stable convergence 
\begin{align} n^{1/2} \Big[\Big(\sum_{i=1}^{m_-^n(t)} (\Delta_i^n C(r))^2 -  \int_0^t \sigma(r)_s^2 ds \Big) + 2 \sum_{i=1}^{m_-^n(t)} \Delta_i^n C(r)\Delta_i^n N(q) \Big] \tols \widetilde V_t+ \widetilde Z_t\,.\end{align}
\end{prop}

Recall that the missing proof of Proposition \ref{propclt} is included as a special case in the previous claim. 

\begin{proof}[Proof of Proposition \ref{propcltb}]
We begin with a proof of the $\mathcal F^0$-stable central limit theorem for realized variance in the continuous case. For fixed $r$, we prove
\begin{align} \label{realvar}
V(r)_t^n := n^{1/2} \Big(\sum_{i=1}^{m_-^n(t)} (\Delta_i^n C(r))^2 -  \int_0^{t} \sigma(r)_s^2 ds \Big) \tols \sqrt 2 \int_0^t \sigma(r)_s^2 (G'(s))^{1/2} dW'_s =: V(r)_t
\end{align}
first. Typically, one rewrites
\[
V(r)_t^n = \sum_{i=1}^{m_-^n(t)} \xi_i^n + o_p(1), \qquad \xi_i^n = n^{1/2}\Big(|\Delta_i^n C(r)|^2 - \int_{t_{i-1}^n}^{t_i^n} \sigma(r)_s^2 \,ds\Big),
\]
using $\E[\pi_n^q] = o(n^{-\alpha})$ from Assumption \ref{scheme}, and then exploits Theorem IX 7.28 in \cite{jacoshir2003} for which several intermediate steps regarding the behaviour of conditional expectations of functionals of $\xi_i^n$ with respect to $\mathcal F^0_{t_{i-1}^n}$ have to be shown. However, as discussed in the proof of Proposition 5.1 in \cite{hayjacyos2011}, one can equally well discuss conditional expectations with respect to $\mathcal G_{i-1}^n = \mathcal F^0_{t_{i-1}^n} \vee \mathcal F^1$. This $\sigma$-algebra represents knowledge of the entire sampling scheme plus knowledge of $X$ up to the stopping time $t_{i-1}^n$. The proof of \eqref{realvar} then boils down to prove 
 \begin{align}
\label{gl1} &\sum_{i=1}^{ m_-^n(t) } \E[\xi_i^n|\mathcal G_{i-1}^n] \stackrel{\mathbb{P}}{\longrightarrow} 0, \\
\label{gl2} &\sum_{i=1}^{m_-^n(t) } \E[(\xi_i^n)^2|\mathcal G_{i-1}^n] \stackrel{\mathbb{P}}{\longrightarrow} 2 \int_0^t \sigma(r)_s^4 G'(s) ds, \\
\label{gl3} &\sum_{i=1}^{m_-^n(t) } \E[\xi_i^n \Delta_i^n M|\mathcal G_{i-1}^n] \stackrel{\mathbb{P}}{\longrightarrow} 0 \text{ for all } M \in \mathcal M, \\
\label{gl4} &\sum_{i=1}^{m_-^n(t) } \E[(\xi_i^n)^4|\mathcal G_{i-1}^n] \stackrel{\mathbb{P}}{\longrightarrow} 0, 
 \end{align}
where $\mathcal M$ is the set of all martingales that are either bounded and orthogonal to $W$ or equal to $W$. Note that (\ref{gl1}) is satisfied identically by construction, whereas (\ref{gl4}) holds using (\ref{e1}) and the assumption on $\pi_n$. The orthogonality condition (\ref{gl3}) follows from standard arguments as e.g.\ laid out in Example 2 of \cite{podovett2010}. What remains to show is thus (\ref{gl2}). First, let $i$ be such that $(k-1) 2^{-r} < t_{i-1}^n < t_i^n < k2^{-r}$ for some $k$. Then
\[\E[(\xi_i^n)^2|\mathcal G_{i-1}^n] = 2n \sigma(r)^4_{(k-1)2^{-r}} (t_i^n - t_{i-1}^n)^2.
\]
Since all other choices of $i$ correspond to at most $2^r$ summands and the length of the intervals between successive observations is bounded by $\pi_n$, we obtain
\[
\sum_{i=1}^{m_-^n(t)} \E[(\xi_i^n)^2|\mathcal G_{i-1}^n] = \sum_{k=1}^{2^r} 2 \sigma(r)^4_{(k-1)2^{-r}} \sum_{i=m_+^n((k-1)2^{-r})}^{m_-^n(k2^{-r})\wedge m_-^n(t)} n (t_i^n - t_{i-1}^n)^2 + o_p(1),
\]
where empty sums are set to be zero. Therefore, since $t_{m_-^n(s)}^n \to s$ for any $s$, the term becomes
\begin{align*}
& \sum_{k=1}^{\lfloor 2^rt \rfloor} 2 \sigma(r)^4_{(k-1)2^{-r}} (G(k2^{-r})-G((k-1)2^{-r})) + 2 \sigma(r)^4_{\lfloor 2^rt \rfloor2^{-r}} (G(t)-G(\lfloor 2^rt \rfloor2^{-r}) + o_p(1) \\ =& 2 \int_0^t \sigma(r)_s^4 \,dG(s) + o_p(1)
\end{align*}
using (\ref{Hexist}), which finishes the proof of (\ref{realvar}) by continuous differentiability of $G$.

On the other hand, we need a joint result regarding $V(r)_t^n$ and the term involving the jumps. Thus, for a fixed $q$, let $S_1, S_2, \ldots$ be a sequence of stopping times exhausting the jumps of $N(q)$. We set $\alpha(n,p) = n^{1/2} (W_{\tau_+^n(S_p)} - W_{\tau_-^n(S_p)})$, and what we discuss in a first step is the $\mathcal F^0$-stable convergence
\begin{align} \label{lemstab1}
(V(r)_t^n,\alpha(n,p)_{p \geq 1}) \tols (V(r)_t,(\eta(S_p)^{1/2} U'_p)_{p \geq 1}), 
\end{align}
as $n \to \infty$. 

The proof is close to the one of Lemma 5.8 in \cite{jacod2008} for most parts, which is why we do not give all details but focus on the intuition behind the steps. Let us begin with several remarks: First, we order the stopping times by the size of the corresponding jumps again. Second, it is enough to prove the result for a fixed $k$, that is we show 
\begin{align} \label{lemstab}
(V(r)_t^n,\alpha(n,p)_{1 \leq p \leq k}) \tols (V(r)_t,(\eta(S_p)^{1/2} U'_p)_{1 \leq p \leq k}) 
\end{align}
only. Formally this means that we have to prove
\begin{align*}
\E\Big[\Psi g(V(r)_t^n) \prod_{p=1}^k h_p(\alpha(n,p))\Big] \to \widetilde \E\Big[\Psi g(V(r)_t) \prod_{p=1}^k h_p(\eta(S_p)^{1/2} U'_p)] 
\end{align*}
for any $\mathcal F^0$-measurable $\Psi$ and for all bounded continuous functions $g,h_1, \ldots, h_k$. However, it is sufficient to focus on $\mathcal G$-measurable variables, which is the $\sigma$-algebra generated by the measure $\mu$ and the processes $b$, $\sigma$, $W$ and $X$, as one may otherwise replace $\Psi$ with $\E[\Psi|\mathcal G]$ and use measurability with respect to $\mathcal G$ of all other variables. 

The first step to obtain the previous relation is to replace $V(r)_t^n$ by some $V(r,\ell)_t^n$ which is defined over those intervals only that do not intersect with those on which the $\alpha(n,p)$ are defined. This approach secures conditional independence of the limiting Brownian motion $W'$ and the normally distributed $U'_p$ later. Precisely, let $S_p^{\ell-}=(S_p-1/\ell)^+$ and $S_p^{\ell+}=S_p+1/\ell$ denote intervals around the jump times and set $B_\ell = \cup_{p \geq 1}^k[S_p^{\ell-},S_p^{\ell+}]$. Let $\Lambda_n(\ell,t)$ denote the set of indices $i$ such that $i \leq m_-^n(t)$ and $[t_{i-1}^n,t_i^n] \cap B_\ell = \emptyset$.  It is rather simple to see that both $V(r)_t^n$ and $V(r,\ell)_t^n=\sum_{i \in \Lambda_n(\ell,t)} \xi_i^n$ and $V(r)_n$ and $V(r,\ell)_n = \sqrt 2 \int_0^t 1_{B_\ell^c}(s) \sigma(r)_s^2 (G'(s))^{1/2} dW'_s$ are in a suitable sense close for large $\ell$, which means that it suffices to prove 
\begin{align*} 
\E\Big[\Psi g(V(r,\ell)_t^n)  \prod_{p=1}^k h_p(\alpha(n,p))\Big] \to \widetilde \E\Big[\Psi g(V(r,\ell)_t)  \prod_{p=1}^k h_p(\eta(S_p)^{1/2} U'_p)\Big]
\end{align*}
for each fixed $\ell$. Fix one such $\ell$ in the following. 
 
Introduce the filtration $\mathcal F^{0,\ell}$ which is the smallest one containing $\mathcal F^{0}$ and is defined in such a way that $W(\ell)_t = \int_0^t 1_{B_\ell}(s) dW_s$ (and the times $S_1, \ldots, S_k$) is $\mathcal F_0^{0,\ell}$-measurable. We will now work conditionally on $\mathcal F_0^{0,\ell}$, so let $Q_\omega$ denote a regular version of this conditional probability. Since $\Delta_i^n W$ is independent of $\mathcal F_0^{0,\ell}$ (for $n$ large enough) and remains normally distributed for $i \in \Lambda_n(\ell,t)$, which contains all but a finite number of indices, reproducing the proof of (\ref{realvar}) yields 
\[
\E_{Q_\omega}[\Psi g(V(r,\ell)_t^n)] \to \widetilde \E_{\tilde Q_\omega}[\Psi g(V(r,\ell)_t)].
\]
Thus, using $\mathcal F_0^{0,\ell}$-measurability of 
$\prod_{p=1}^k h_p(\alpha(n,p))$, we obtain 
\begin{align*}
\E\big[\Psi g(V(r,\ell)_t^n) \prod_{p=1}^k h_p(\alpha(n,p))\big] &= \E\big[\E_{Q_\omega}[\Psi g(V(r,\ell)_t^n)] \prod_{p=1}^k h_p(\alpha(n,p)) \big] \\ &\sim \E\big[\E_{\tilde Q_\omega}[\Psi g(V(r,\ell)_t)] \prod_{p=1}^k h_p(\alpha(n,p)) \big].
\end{align*}
Since $\E_{\tilde Q_\omega}[\Psi g(V(r,\ell)_t)]$ is $\mathcal F_0^{0,\ell}$-measurable again, everything then boils down to prove 
\begin{align*} 
\E\big[\Gamma \prod_{p=1}^k h_p(\alpha(n,p))\big] \to \widetilde \E\big[\Gamma \prod_{p=1}^k h_p(\eta(S_p)^{1/2} U'_p)\big].
\end{align*}
for any $\mathcal F_0^{0,\ell}$-measurable $\Gamma$. By conditioning, we may again restrict us to those variables generated by the jump times of $N(q)$ and the Brownian motion $W(\ell)$, and using Lemma 2.1 in \cite{jacodprotter} this means to prove
\begin{align*}
\E\big[ \gamma(W(\ell))  \kappa(S_1, \ldots, S_k)\prod_{p=1}^k h_p(\alpha(n,p))\big] \to \widetilde \E\big[ \gamma(W(\ell))  \kappa(S_1, \ldots, S_k)\prod_{p=1}^k h_p(\eta(S_p)^{1/2} U'_p)\big]
\end{align*}
for any bounded continuous functions $ \gamma$ and $ \kappa$. Since $W(\ell)$ is independent of any other quantity involved, this means finally that we have to prove
\[
\E\Big[ \kappa(S_1,\ldots,S_k) \prod_{p=1}^k h_p(\alpha(n,p))\Big] \to \E\big[ \kappa(S_1,\ldots,S_k) \prod_{p=1}^k h_p(\eta(S_p)^{1/2} U'_p)\big].
\]
On the set $\Omega(q,r,n,t)$, introduced in the final step of the proof of Proposition \ref{propremainder}, the variables $\alpha(n,p)$ are defined over non-overlapping intervals which explains independence of the limiting variables. Since $\Omega(q,r,n,t) \to \Omega$ for $n \to \infty$, we may add (and subtract again) an indicator function over $\Omega(q,r,n,t)$ on the left hand side above. Then, 
\[
\E\big[ \kappa(S_1,\ldots,S_k) \prod_{p=1}^k h_p(\alpha(n,p))\big] \sim \widetilde \E\big[ \kappa(S_1,\ldots,S_k) \prod_{p=1}^k h_p((n(\tau^n_+(S_p) - \tau^n_-(S_p)))^{1/2} U'_p)\big].
\]
We are therefore left to prove the stable convergence of $(n(\tau^n_+(S_p) - \tau^n_-(S_p)))_{1 \leq p \leq k}$ to $(\eta(S_p))_{1 \leq p \leq k}$ which is exactly condition \eqref{locexist} using the joint uniform distribution of the jump times of $N(q)$.

Finally, on $\Omega(q,r,n,t)$ all of the jumps occur on intervals of piecewise constancy of $\sigma(r)$. Thus we obtain easily
\[
V(r)_t^n + 2n^{1/2} \sum_{i=1}^{m^n_-(t)} \Delta_i^n C\Delta_i^n N(q) \tols V(r)_t + 2 \sum_{p} \Delta N(q)_{S_p} \sigma(r)_{S_p} \eta(S_p)^{1/2} U'_p.
\]
The proof can be finished by first letting $r \to \infty$ and then $q \to \infty$: Since $\sigma$ is continuous, we have both $\WEE[|V(r)_t - V_t|^2] \to 0$ and
\[
\sum_{p} \WEE[|\Delta N(q)_{S_p}| |\sigma(r)_{S_p} - \sigma_{S_p}| |\eta(S_p)^{1/2}| |U'_p|] \leq K \sum_{p} \E[1_{\{S_p \leq t\}} |\sigma(r)_{S_p} - \sigma_{S_p}|] \to 0
\]
by successive conditioning and boundedness of the first moments of $\eta(S_p)$ and $\Pi(q)$. Finally, 
\[
\WEE\Big[\Big|Z_t - \sum_{p} \Delta N(q)_{S_p} \sigma_{S_p} \eta(S_p)^{1/2} U'_p\Big|^2\Big] \leq K \E\Big[ \sum_{t} |\Delta X_s|^2 1_{|\Delta X_s| \leq 1/q} \Big] \to 0,
\]
again from Lebesgue's theorem. 
%Since $\sigma$ is continuous, we have $\WEE[|V(r)_t - V_t|^2] \to 0$, and also
%\[
%\sum_{p} \WEE[|\Delta N(q)_{S_p}| |\sigma(r)_{S_p} - \sigma_{S_p}| |\eta(S_p)^{1/2}| |U'_p|] \leq K \sum_{p} \E[1_{\{S_p \leq t\}} |\sigma(r)_{S_p} - \sigma_{S_p}|^4]^{1/4},
%\]
%where we have used boundedness of the jumps, Cauchy-Schwarz inequality twice and boundedness of $\WEE[(U'_p)^4]$ and $\WEE[\eta(S_p)]$, the latter being is satisfied by Assumption \ref{scheme}. $\E[(\Pi(q))^2] \leq K$, continuity of $\sigma$ and Lebesgue's theorem yield convergence to zero of the term above. Finally, 
%\[
%\WEE\Big[\Big|Z_t - \sum_{p} \Delta N(q)_{S_p} \sigma_{S_p} \eta(S_p)^{1/2} U'_p\Big|^2\Big] \leq K \E\Big[ \sum_{t} |\Delta X_s|^2 1_{|\Delta X_s| \leq 1/q} \Big] \to 0
%\]
%from Lebesgue's theorem. 
\end{proof}

\section{Proof of Theorem \ref{theo2} \label{app3}}
\setcounter{equation}{0}
\renewcommand{\theequation}{\thesection.\arabic{equation}} 
We decompose $X^{(1)}$ and $X^{(2)}$ in the same way as above and denote the single addends likewise. The first part of the proof establishes the limit theorem for the leading variance term:
\begin{prop}\label{propclt2}On the assumptions of Theorem \ref{theo2},
the term
\begin{align}&\notag\sqrt{n}\Big(\sum_{t_i^{(1)}\le t}\sum_{t_j^{(2)}\le t}\Big(\Delta_i	^{n_1}C^{(1)}(r)\Delta_j^{n_2}C^{(2)}(r)+\Delta_j^{n_2}N^{(2)}(q)\Delta_i^{n_1}C^{(1)}(r)+\Delta_i^{n_1}N^{(1)}(q)\Delta_j^{n_2}C^{(2)}(r)\Big)\Big. \\
&\left.\quad \label{step1t} \times 1_{\{\min{(t_{i}^{(1)},t_{j}^{(2)})}>\max{(t_{i-1}^{(1)},t_{j-1}^{(2)})}\}}-\int_0^t(\rho\sigma^{(1)}\sigma^{(2)})(r)_s\,ds\right)\end{align}
satisfies:
\begin{align}\label{step1prop2}\eqref{step1t}\tols \tilde {\mathcal{V}}_t+\tilde {\mathcal{Z}}_t\,.\end{align}
\end{prop}
\begin{proof}
The claim draws on the stable central limit theorem for the continuous semimartingale as one building block. For the local parametric approximation with fixed $r$
\begin{align*}\tilde {\mathcal{V}}(r)_t^n:=\sqrt{n}\Big(\sum_{t_i^{(1)}\le t}\sum_{t_j^{(2)}\le t}(\Delta_i	^{n_1}C^{(1)}(r)\Delta_j^{n_2}C^{(2)}(r))1_{\{\min{(t_{i}^{(1)},t_{j}^{(2)})}>\max{(t_{i-1}^{(1)},t_{j-1}^{(2)})}\}}-\int_0^t(\rho\sigma^{(1)}\sigma^{(2)})(r)_s\,ds\Big),\end{align*}
we have that $\tilde {\mathcal{V}}(r)_t^n\tols \tilde {\mathcal{V}}(r)_t$; see \cite{hayayosh2011} and \cite{asyn}. Thus, the proof affiliates to the proof of Theorem \ref{theo1b} based on an extension of Lemma 5.8 in \cite{jacod2008}. 
%For a fixed integer $q$, consider the set $Q=\{z\in\R^2|\gamma(z)>q^{-1}\}$ and note that $1_Q*\mu$ is a Poisson process with parameter $\lambda(Q)$, where $\lambda$ denotes the Lebesgue measure. The jump times of $Q$ are denoted $S_p,1\le p\le \Pi(q)_t$, and $1_Q*\mu$ takes the value $\Pi(q)_t$ at time $t\in[0,1]$. The set $\{S_p\}$ contains co-jumps and idiosyncratic jumps. We restrict to the set
The only major difference regards the terms involving the (co-)jumps, as more quantities are of interest now. Set
\begin{align*}\alpha(n,p)&=n^{\nicefrac12}\Big(\Big(W_{\tau_{++}^{(1,2)}(S_p)}^{(1)}-W^{(1)}_{\max{(\tau_+^{1}(S_p),\tau_+^2(S_p))}}\Big),
\Big(W^{(1)}_{\min{(\tau_-^{(1)}(S_p),
\tau_-^2(S_p))}}-W_{\tau_{--}^{(1,2)}(S_p)}^{(1)}\Big),\\
&\quad \Big(W_{\tau_{++}^{(2,1)}(S_p)}^{(2)}-W^{(2)}_{\max{(\tau_+^{1}(S_p),\tau_+^2(S_p))}}\Big), \Big(W^{(2)}_{\min{(\tau_-^{(1)}(S_p),
\tau_-^2(S_p))}}-W_{\tau_{--}^{(2,1)}(S_p)}^{(2)}\Big),\\
&\quad \Big(W^{(1)}_{\max{(\tau_+^{(1)}(S_p),\tau_+^2(S_p))}}-W^{(1)}_{\min{(\tau_-^{(1)}(S_p),
\tau_-^2(S_p))}}\Big),
\Big(W^{(2)}_{\max{(\tau_+^{(1)}(S_p),\tau_+^2(S_p))}}-W^{(2)}_{\min{(\tau_-^{(1)}(S_p),\tau_-^2(S_p))}}\Big)
\Big),
\end{align*}
motivated in \eqref{rl1}--\eqref{l3} above. The convergence of $\alpha(n,p)$ to mixtures of independent normal limiting variables and the joint convergence follow in an analogous way as for Theorem \ref{theo1b} then, using \eqref{locexist2mult} to establish stable convergence of the corresponding lengths of the intervals.
 \end{proof}

\begin{prop}\label{propremainder2} On the assumptions of Theorem \ref{theo2}:
\begin{align}\label{propd2}
\lim_{q \to \infty} \limsup_{r \to \infty} \limsup_{n \to \infty} \sqrt{n}\left|\widehat{\left[ X^{(1)}, X^{(2)}\right]}_t^{(HY),n}-\eqref{step1t}\right|\stackrel{\PP}{\longrightarrow}0
\end{align}
and, moreover, the right-hand side of \eqref{propd2}/$\sqrt{n}$ tends to zero in probability on milder assumptions as long as the mesh $\pi_n\rightarrow 0$. 
\end{prop}

\begin{proof}Based on the decomposition of $X$ as before, the terms are treated analogously as in the proof of Theorem \ref{theo1}. Most upper bounds can be deduced along the same lines, with the exception that based on the illustration \eqref{varHY3} one has also to consider interpolated terms and dependence of adjacent addends. Yet, when denoting $\Delta_k^+, \Delta_k^-$ the interpolation intervals which are non-zero, we may employ the simple estimate
\[\Delta_k^2+\Delta_k\Delta_k^++\Delta_k\Delta_k^-+\Delta_k^+\Delta_k^-\le \Delta_k^2+\Delta_{k+1}\Delta_k+\Delta_{k-1}\Delta_k+\Delta_{k-1}\Delta_{k+1}\,,\]
where $\Delta_{k+1}=T_{k+1}-T_k$ are refresh time instants as before. After an application of the Cauchy-Schwarz inequality and/or measurability arguments the addends of the remainder terms have the same structure as in the synchronous case. Hence, consistency of the HY-estimator and the CLT readily follow from the standard estimates \eqref{e1}--\eqref{e5} with the strategy of proof from the proof of Theorem \ref{theo1}.
\end{proof}
\bibliography{jumps}

\end{document}